\documentclass[10pt]{article}

\usepackage{latexsym, amssymb, amsmath, amsthm, amscd}
\usepackage[all,cmtip]{xy}
\usepackage{amsmath}
\usepackage{amsthm}
\usepackage{amssymb}
\usepackage{amsfonts}
\usepackage{amsrefs}
\usepackage{color,authblk}
\usepackage{tikz}
\usepackage{amscd} 
\usepackage{comment}
\usepackage{mathrsfs}
\usepackage[all]{xy}
\usepackage{graphicx}
\usepackage{authblk}
\usepackage{hyperref}
\usepackage{fullpage}
\usepackage[all,cmtip]{xy}
\usepackage{amsmath,amscd}
\usepackage{tikz-cd}
\usepackage{tikz}
\usetikzlibrary{matrix}

\usepackage{hyperref}

\numberwithin{equation}{section} \DeclareMathSizes{2}{10}{12}{13}
\newtheorem{thm}{Proposition}[section]
\newtheorem{Thm}[thm]{Theorem}

\newtheorem{cor}[thm]{Corollary}
\newtheorem{lem}[thm]{Lemma}
\newtheorem{defn}[thm]{Definition}

\title{Weak comp algebras and cup products in secondary Hochschild cohomology of entwining structures}

\author{Mamta Balodi \footnote{Department of Mathematics, Indian Institute of Science, Bangalore, India. Email : mamta.balodi@gmail.com} $\qquad$ Abhishek Banerjee \footnote{Department of Mathematics, Indian Institute of Science, Bangalore, India. Email : abhishekbanerjee1313@gmail.com} \footnote{AB was partially supported by SERB Matrics fellowship MTR/2017/000112} $\qquad$ Anita Naolekar \footnote{Theoretical Statistics and Mathematics Unit,  Indian Statistical Institute, Bangalore, India. Email : anita.naolekar@gmail.com}}

\date{}

\begin{document}

\maketitle 

\medskip

\begin{abstract} We define the secondary Hochschild complex for an entwining structure over a commutative $k$-algebra $B$. We show that this complex
carries the structure of a weak comp algebra. We obtain two distinct cup product structures  for the secondary cohomology groups. We also consider a subcomplex
on which the two cup products coincide and which satisfies the axioms for being a comp algebra. The  cohomology of this subcomplex then forms a Gerstenhaber algebra. We also 
construct a bicomplex that controls the deformations of the entwining structure over $B$.
\end{abstract}

\medskip
MSC(2010) Subject Classification:   16E40, 16W30.


\section{Introduction}

\medskip
Let $k$ be a field and let $B$ be a commutative $k$-algebra. Let $A$ be a $k$-algebra and $\zeta: B\longrightarrow A$ be a morphism of $k$-algebras
such that the image of $\zeta$ lies in the center of $A$. In \cite{St1}, Staic introduced Hochschild cohomology groups for such a datum 
$(A,B,\zeta)$ with coefficients in an $A$-bimodule $M$, which he referred to as the ``secondary Hochschild cohomology'' $HH^\bullet ((A,B,\zeta);M)$. In \cite{St1}, it was shown that the secondary cohomology controls the deformations of a family of products $\{\mu_b:A\otimes A\longrightarrow A\}_{b\in B}$ on $A$ satisfying a generalized associativity condition.
In later papers (see Staic and Stancu \cite{St2}, Corrigan-Salter and Staic \cite{St3}, Laubacher, Staic and Stancu \cite{St4}), it was revealed that the secondary cohomology  groups carry a very rich structure, as a result of which they may be treated as an object of study in their  own right. In \cite{BBN}, we have studied Batalin-Vilkovisky operators on secondary  cohomology groups. The purpose of this paper is to introduce and study secondary Hochschild cohomology groups for an entwining structure over a commutative $k$-algebra $B$. 

\smallskip
 The concept of an entwining structure,  introduced by Brzezi\'{n}ski and Majid \cite{Brz1}, has been developed widely in the literature as a unifying formalism for studying multiple theories such as
relative Hopf modules, Doi-Hopf and Yetter-Drinfeld modules as well as coalgebra Galois extensions (see, for instance, \cite{Abu}, \cite{BBR},  \cite{Brz2.5}, \cite{Brz3}, \cite{BCT},   \cite{CaDe}, \cite{Jia}, \cite{Schbg}). An entwining structure $(A,C,\psi)$ consists of an algebra $A$, a coalgebra $C$ and a map $\psi: C\otimes A\longrightarrow A\otimes C$ satisfying certain conditions. In particular, such a datum
$(A,C,\psi)$ behaves in many ways like a bialgebra or more generally like a comodule algebra over a bialgebra and Brzezi\'{n}ski introduced a Hochschild type
cohomology theory for entwining structures in \cite{Brz2}. The fundamental idea that underlies Brzezi\'{n}ski's construction in \cite{Brz2} is his notion
of a weak comp algebra, introduced by generalizing comp algebras \cite{GS2} or pre-Lie systems \cite{G1}, \cite{GerSh} (see also the exposition in \cite{KPS}). This leads to the notion of not one but two natural cup product structures on the Hochschild complex of an entwning structure, both of which descend to the level of cohomology. 

\smallskip
Accordingly, we are motivated in this paper to study a secondary cohomology theory for entwining structures over a given commutative $k$-algebra $B$. This consists of an entwining structure $(A,C,\psi)$ along with a map $\zeta : B\longrightarrow A$ of algebras such that $\psi(c\otimes \zeta(b))=\zeta (b)\otimes c$ for every $b\in B$ and $c\in C$. We introduce a secondary Hochschild complex $\mathscr C^\bullet_\psi((A,B,C,\zeta);M)$ with coefficients in an $A$-bimodule $M$ satisfying certain conditions. When $M=A$, this leads to
a secondary Hochschild complex $\mathscr C^\bullet_\psi(A,B,C,\zeta)$   on which we study comp structures and cup products  in a manner similar to the complex of Brzezi\'{n}ski \cite{Brz2}.
 The cohomology groups of $\mathscr C^\bullet_\psi(A,B,C,\zeta)$  will be denoted by $HH_\psi^\bullet(A,B,C,\zeta)$.  We also combine $\mathscr C^\bullet_\psi((A,B,C,\zeta);M)$  with the Cartier complex of a coalgebra in a manner similar to \cite{Brz2} to obtain a bicomplex which controls the deformations of the entwining structure
 $(A,B,C,\psi,\zeta)$ over $B$.
 
 \smallskip We mention that the secondary Hochschild complex of Staic \cite{St1} admits a canonical inclusion into  $\mathscr C^\bullet_\psi((A,B,C,\zeta);M)$. This induces a morphism from the secondary cohomology of Staic \cite{St1} to the secondary cohomology of the entwining structure over $B$ with coefficients in $M$. Our first main result is as follows. 
 
\begin{Thm} (see Theorem \ref{Th3} and Corollary \ref{C3.8}) Let $B$ be a commutative $k$-algebra and let $(A,B,C,\psi,\zeta)$ be an entwining structure over $B$.  Then, 

\smallskip
(a) $\mathscr C_\psi^\bullet(A,B,C,\zeta)$ carries the structure of a weak comp algebra.

\smallskip
(b) There are cup products on the secondary Hochschild cohomology of the entwining structure
\begin{equation*}
\begin{array}{c}
\cup : HH^m_\psi(A,B,C,\zeta)\otimes HH_\psi^n(A,B,C,\zeta) \longrightarrow HH_\psi^{m+n}(A,B,C,\zeta) \\ \sqcup: HH^m_\psi(A,B,C,\zeta)\otimes HH_\psi^n(A,B,C,\zeta) \longrightarrow HH_\psi^{m+n}(A,B,C,\zeta)
\end{array}
\end{equation*} which are related as follows: for cohomology classes $\bar{f}\in HH^m_\psi(A,B,C,\zeta)$ and $\bar{g}\in HH^n_\psi(A,B,C,\zeta)$, we have
\begin{equation*}
\bar{f}\cup \bar{g}=(-1)^{mn}\bar{g}\sqcup \bar{f} 
\end{equation*}

\end{Thm}  

As such, we consider a    subcomplex $\mathscr E^\bullet_\psi(A,B,C,\zeta)\subseteq \mathscr C^\bullet_\psi(A,B,C,\zeta)$ on which the two cup products coincide, inducing a graded commutative structure on the  cohomology groups. Moreover, we show that this graded commutative product forms part of a Gerstenhaber algebra structure on the cohomology 
of  $\mathscr E^\bullet_\psi(A,B,C,\zeta)$.

\begin{Thm} (see Theorem \ref{T4.3} and  Theorem \ref{thfn})   Let $B$ be a commutative $k$-algebra and let $(A,B,C,\psi,\zeta)$ be an entwining structure over $B$. Then,

(a) The cup products $\cup$ and $\sqcup$ on the complex  $\mathscr E_\psi^\bullet(A,B,C,\zeta)$ coincide.

\smallskip
(b)  $\mathscr E_\psi^\bullet(A,B,C,\zeta)$ is  a right comp algebra over $k$.
 
\smallskip
(c) The cohomology $(H^\bullet(\mathscr E_\psi^\bullet(A,B,C,\zeta)),\cup,[.,.])$ carries the structure of a Gerstenhaber algebra. 

\end{Thm}

Finally, in Section 5, we come to the deformations of  the entwining structure
 $(A,B,C,\psi,\zeta)$ over $B$. A deformation $(\mu_t,\Delta_t,\psi_t)$ of  $(A,B,C,\psi,\zeta)$ over $B$ (see Definition \ref{Dinfdef}) will consist of three components:
 
 \smallskip
 (a) A deformation of a family of products on $A$ parametrized by elements of $B$ and  satisfying a generalized associativity condition. 
 
 \smallskip
 (b) A coassociative deformation of the coalgebra $C$.
 
 \smallskip
 (c)  A deformation of the entwining map $\psi$ satisfying certain compatibility conditions.
 
 \smallskip
 We consider on the one hand the secondary Hochschild complexes $\mathscr C^\bullet_\psi((A,B,C,\zeta);A\otimes C^{\otimes k})$ with coefficients in the $A$-bimodules
 $A\otimes C^{\otimes k}$. On the other hand, we consider the complexes determining the entwined cohomology of $C$ with coefficients in the $C$-bicomodules 
 $C\otimes A^{\otimes l}\otimes B^{\otimes \frac{l(l-1)}{2}}$. We combine them to form a bicomplex $C^{\bullet, \bullet}(A,B,C,\psi,\zeta)$ whose total cohomology is denoted by $H^\bullet(A,B,C,\psi,\zeta)$. Our final result is as follows.

\begin{Thm} (see Theorems \ref{MThinf} and \ref{obsthmi}) 
Let $(A,B,C,\psi,\zeta)$ be an entwining structure over $B$. 

\smallskip
(a) There is a one-to-one correspondence between equivalence classes of infinitesimal deformations of  $(A,B,C,\psi,\zeta)$ and the cohomology group $H^2(A,B,C,\psi,\zeta)$.

\smallskip
(b) Let $(\mu_t,\Delta_t,\psi_t)$ be a  deformation of    $(A,B,C,\psi,\zeta)$ modulo $t^{n+1}$. Then, $(\mu_t,\Delta_t,\psi_t)$ can be lifted to a deformation
modulo $t^{n+2}$ if and only if the obstruction is a coboundary, i.e.,
$
Obs^{n+1} =0\in H^3(A,B,C,\psi,\zeta)
$.

\end{Thm}
 
 \smallskip

{\bf Acknowledgements:} We are grateful to Tomasz Brzezi\'{n}ski for useful discussions on the deformation complex of an entwining structure.

\smallskip

\section{Secondary Hochschild cohomology of an entwining structure}

\medskip
Let $k$ be a field. Throughout, we let $A$  be an algebra over $k$ having product structure $\mu:A\otimes A\longrightarrow A$ and unit map $u:k\longrightarrow A$. Also, we let
$C\ne 0$ be a $k$-coalgebra, with coproduct structure $\Delta: C\longrightarrow C\otimes C$ and counit map $\varepsilon : C\longrightarrow k$. We will generally suppress the summation in Sweedler notation and simply write $\Delta(c)=c_1\otimes c_2$ for any $c\in C$. For a map $\psi: C\otimes A\longrightarrow A\otimes C$, we will also suppress the summation and write $\psi(c\otimes a)=a_\psi\otimes c^\psi$ for any $a\in A$ and $c\in C$.

\smallskip We begin by recalling the notion of an entwining structure over $k$, which was introduced  by Brzezi\'{n}ski and Majid \cite{Brz1}. 

\begin{defn}\label{D2.1} Let $k$ be a field. An entwining structure over $k$ is a triple $(A,C,\psi)$ consisting of the following data

\smallskip
(1) An algebra $A$ over $k$.

\smallskip
(2) A coalgebra $C$ over $k$.

\smallskip
(3) A $k$-linear morphism $\psi: C\otimes A\longrightarrow A\otimes C$ satisfying the following conditions for any $a$, $b\in A$ and $c\in C$.
\begin{equation}\label{ent}
\begin{array}{c}
\psi(c\otimes \mu(a\otimes b))=\psi(c\otimes ab)=(ab)_\psi\otimes c^\psi = a_\psi b_\psi\otimes {c^\psi}^\psi=((\mu\otimes 1)\circ (1\otimes \psi)\circ (\psi\otimes 1))(c\otimes a\otimes b)\\
(1\otimes \Delta)(\psi (c\otimes a))=a_\psi\otimes \Delta(c^\psi)={a_{\psi}}_{\psi}\otimes c_1^\psi \otimes c_2^\psi =((\psi\otimes 1)\circ (1\otimes \psi))(\Delta(c)\otimes a) \\
a_\psi\varepsilon (c^\psi)=\varepsilon(c)a \qquad 1_\psi\otimes c^\psi = 1\otimes c\\
\end{array}
\end{equation}

\end{defn}

Let $B$ be a commutative $k$-algebra and let $\zeta: B\longrightarrow A$ be a morphism of $k$-algebras such that the image of $\zeta$ lies in the center $Z(A)$ of $A$, i.e., 
$\zeta(B)\subseteq Z(A)$. Let $M$ be an $A$-bimodule such that $\zeta(b)m=m\zeta(b)$ for each $b\in B$ and $m\in M$. The `secondary Hochschild cohomology' of such a triple $(A,B,\zeta)$ with coefficients in $M$ was introduced by Staic \cite{St1} and studied further in a series
of papers \cite{St2}, \cite{St3}, \cite{St4}. When $B=k$ and $\zeta$ is the unit map of the $k$-algebra $A$, this reduces to the ordinary Hochschild cohomology $HH^\bullet(A,M)$. 

\begin{defn}\label{D2.2} Let $B$ be a commutative $k$-algebra. An entwining structure over $B$ is a tuple $(A,B,C,\psi,\zeta)$ consisting of the following data

\smallskip
(1) $\zeta: B\longrightarrow A$ is a morphism of $k$-algebras such that  
$\zeta(B)\subseteq Z(A)$.

\smallskip
(2) The triple $(A,C,\psi)$ is an entwining structure over $k$ satisfying the additional condition that  
\begin{equation}\label{eq2.2} \zeta(b)_\psi\otimes c^\psi=\psi(c\otimes \zeta(b))=\zeta(b)\otimes c
\end{equation}
for each $b\in B$ and $c\in C$.

\end{defn}

For an $A$-bimodule $M$, the Hochschild complex of an entwining structure $(A,C,\psi)$ with coefficients in $M$ was introduced by Brzezi\'{n}ski in \cite[$\S$ 2]{Brz2}. From now onwards we will always suppose that $M$ satisfies the following condition: 
\begin{equation}\label{sym}
\zeta(b)m=m\zeta(b) \qquad \forall\textrm{ }b\in B, m\in M
\end{equation} We are now ready to introduce the secondary Hochschild complex $\mathscr C^\bullet_\psi((A,B,C,\zeta);M)$ of an entwining structure $(A,B,C,\psi,\zeta)$ over $B$ with coefficients in $M$. We set
\begin{equation}\label{secondary}
\mathscr C^n_\psi((A,B,C,\zeta);M)=Hom_k(C\otimes A^{\otimes n}\otimes B^{\otimes \frac{n(n-1)}{2}},M)
\end{equation} For the sake of convenience, we will write an element of $C\otimes A^{\otimes n}\otimes B^{\otimes \frac{n(n-1)}{2}}$ as a ``tensor matrix''
\begin{equation}\label{tensmat}
c\otimes M=c\otimes ((m_{ij}))_{1\leq i,j,\leq n}=c\bigotimes \begin{pmatrix}
a_1 & b_{12} & b_{13} & b_{14} & ...& b_{1,n-2} & b_{1,n-1} & b_{1n} \\
1 & a_2 & b_{23}  & b_{24} &  ...& b_{2,n-2} & b_{2,n-1} & b_{2n} \\
1 & 1 & a_3  & b_{34} &  ...& b_{3,n-2} & b_{3,n-1} & b_{3n} \\
. & . & .  & . &  ...& . & . & .\\
1 & 1 & 1 & 1 &  ...& 1 & a_{n-1} & b_{n-1,n}\\
1 & 1 & 1 & 1 &  ...& 1 & 1 & a_n \\
\end{pmatrix}
\end{equation} with $c\in C$, $a_i\in A$, $b_{ij}\in B$ and $1\in k$.  We need to describe the differential
\begin{equation}
\delta^n : \mathscr C^n_\psi((A,B,C,\zeta);M)\longrightarrow \mathscr C^{n+1}_\psi((A,B,C,\zeta);M)
\end{equation} For this, we consider $f\in \mathscr C^n_\psi((A,B,C,\zeta);M)=Hom_k(C\otimes A^{\otimes n}\otimes B^{\otimes \frac{n(n-1)}{2}},M)$. Then, we define $\delta(f)=\delta^n(f)\in 
Hom_k(C\otimes A^{\otimes n+1}\otimes B^{\otimes \frac{(n+1)n}{2}},M)$ as follows

\begin{equation}\label{maindiff}
\begin{array}{l}
\delta^n(f)\left(c\bigotimes \begin{pmatrix}
a_1 & b_{12} & b_{13} & b_{14} & ...& b_{1,n-1} & b_{1,n} & b_{1,n+1} \\
1 & a_2 & b_{23}  & b_{24} &  ...& b_{2,n-1} & b_{2,n} & b_{2,n+1} \\
1 & 1 & a_3  & b_{34} &  ...& b_{3,n-1} & b_{3,n} & b_{3,n+1} \\
. & . & .  & . &  ...& . & . & .\\
1 & 1 & 1 & 1 &  ...& 1 & a_{n} & b_{n,n+1}\\
1 & 1 & 1 & 1 &  ...& 1 & 1 & a_{n+1}\\
\end{pmatrix}\right) \\\\
= \zeta\left( \underset{j=2}{\overset{n+1}{\prod}}b_{1j}\right) a_{1\psi}\cdot f\left(c^\psi\bigotimes \begin{pmatrix}
 a_2 & b_{23}  & b_{24} &  ...& b_{2,n-1} & b_{2,n} & b_{2,n+1} \\
 1 & a_3  & b_{34} &  ...& b_{3,n-1} & b_{3,n} & b_{3,n+1} \\
 . & .  & . &  ...& . & . & .\\
 1 & 1 & 1 &  ...& 1 & a_{n} & b_{n,n+1}\\
 1 & 1 & 1 &  ...& 1 & 1 & a_{n+1} \\
\end{pmatrix}\right)
\end{array}
\end{equation}
\begin{equation*}
\begin{array}{l}
+\underset{i=1}{\overset{n}{\sum}}(-1)^if\left(c\bigotimes \begin{pmatrix}
a_1 & b_{12} & \dots & b_{1i}b_{1,i+1} & \dots & b_{1,n+1}\\
1 & a_2   & \dots & b_{2i}b_{2,i+1} & \dots & b_{2,n+1} \\
. & .   & \dots &\dots & \dots & . \\
1 & 1  & \dots & \zeta(b_{i,i+1})a_{i}a_{i+1} &  \dots & b_{i,n+1}b_{i+1,n+1}\\
. & .  & \dots & \dots & \dots & . \\
1 & 1  & \dots &\dots &  \dots & b_{n,n+1}\\
1 & 1  & \dots & \dots &\dots & a_{n+1} \\
\end{pmatrix}\right) \\\\
+(-1)^{n+1}\zeta\left(\underset{k=1}{\overset{n}{\prod}}b_{k,n+1}\right)f\left(c\bigotimes \begin{pmatrix}
a_1 & b_{12} & b_{13} & b_{14} & ...& b_{1,n-1} & b_{1,n}  \\
1 & a_2 & b_{23}  & b_{24} &  ...& b_{2,n-1} & b_{2,n}   \\
1 & 1 & a_3  & b_{34} &  ...& b_{3,n-1} & b_{3,n}   \\
. & . & .  & . &  ...& . & . \\
1 & 1 & 1 & 1 &  ...& 1 & a_{n}  \\
\end{pmatrix}\right)\cdot a_{n+1}\\
\end{array}
\end{equation*} We need to show that $\delta^{n+1}\circ \delta^n=0$. For this, we use the canonical isomorphisms
\begin{equation} \label{thetcan}
\begin{array}{c}
\begin{CD} \theta^n : \mathscr C^n_\psi((A,B,C,\zeta);M)=Hom(C\otimes A^{\otimes n}\otimes B^{\otimes \frac{n(n-1)}{2}},M) @>\cong >> Hom( A^{\otimes n}\otimes B^{\otimes \frac{n(n-1)}{2}},Hom(C,M))\\ \end{CD} \\ \\
 \theta^n(f)\begin{pmatrix}
a_1 & b_{12} & b_{13} & b_{14} & ...& b_{1,n-1} &   b_{1n} \\
1 & a_2 & b_{23}  & b_{24} &  ...& b_{2,n-1} &  b_{2n} \\
1 & 1 & a_3  & b_{34} &  ...& b_{3,n-1} &  b_{3n} \\
. & . & .  & . &  ...& .   & .\\
1 & 1 & 1 & 1 &  ...& 1 & a_n\\
\end{pmatrix}(c)=f\left(c\otimes \begin{pmatrix}
a_1 & b_{12} & b_{13} & b_{14} & ...& b_{1,n-1} &   b_{1n} \\
1 & a_2 & b_{23}  & b_{24} &  ...& b_{2,n-1} &  b_{2n} \\
1 & 1 & a_3  & b_{34} &  ...& b_{3,n-1} &  b_{3n} \\
. & . & .  & . &  ...& .   & .\\
1 & 1 & 1 & 1 &  ...& 1 & a_n\\
\end{pmatrix} \right)\\
 \end{array}
\end{equation} We now recall (see \cite[$\S$ 2]{Brz2}) that $Hom(C,M)$ becomes  an $A$-bimodule by setting
\begin{equation}\label{eq2.6}
(g\cdot a)(c)=g(c) \cdot a \qquad (a\cdot g)(c)=a_\psi\cdot g(c^\psi)
\end{equation} for any $g\in Hom(C,M)$, $a\in A$ and $c\in C$.

\begin{thm}\label{P2.3} Let $B$ be a commutative $k$-algebra and let $(A,B,C,\psi,\zeta)$ be an entwining structure over $B$. Let $M$ be an $A$-bimodule such that $\zeta(b)m=m\zeta (b)$ 
for each $b\in B$ and $m\in M$. Then, $(\mathscr C_\psi^\bullet((A,B,C,\zeta);M),\delta^\bullet)$ is a cochain complex, i.e., $\delta\circ \delta=0$. 

\end{thm}

\begin{proof}Since $(A,B,C,\psi,\zeta)$ is an entwining structure over $B$, it follows from \eqref{eq2.2} and \eqref{sym} that
\begin{equation}\label{eq2.7}
(\zeta(b)\cdot g)(c)=\zeta(b)_\psi \cdot g(c^\psi) =\zeta(b)\cdot g(c)=g(c)\cdot \zeta(b)=(g\cdot \zeta(b))(c)
\end{equation} for any $b\in B$, $c\in C$ and $g\in Hom(C,M)$. 
From \eqref{eq2.7}, it now follows that we can construct as in \cite[Proposition 2]{St1} a secondary Hochschild complex 
\begin{equation*}(\mathscr C^\bullet((A,B,\zeta);Hom(C,M)),\delta')
\end{equation*}for the triple $(A,B,\zeta)$  with coefficients in the $A$-bimodule $Hom(C,M)$. The terms $Hom( A^{\otimes n}\otimes B^{\otimes \frac{n(n-1)}{2}},Hom(C,M))$ of this complex are canonically identified with
those of $\mathscr C_\psi^\bullet((A,B,C,\zeta);M)$. In order to show that $(\mathscr C_\psi^\bullet((A,B,C,\zeta);M),\delta^\bullet)$ is a cochain complex, it suffices therefore to check that the differential $\delta$ can be identified with $\delta'$. 

\smallskip
Using \eqref{eq2.6} and maintaining the notation in \eqref{maindiff}, we see that

\begin{equation*}
\tiny
 \zeta\left( \underset{j=2}{\overset{n+1}{\prod}}b_{1j}\right) a_{1\psi}\cdot f\left(c^\psi\bigotimes \begin{pmatrix}
 a_2 & b_{23}  & b_{24} &  ...&   b_{2,n} & b_{2,n+1} \\
 1 & a_3  & b_{34} &  ...&   b_{3,n} & b_{3,n+1} \\
 . & .  & . &  ...&  . & .\\
 1 & 1 & 1 &  ...  & a_{n} & b_{n,n+1}\\
 1 & 1 & 1 &  ...  & 1 & a_{n+1} \\
\end{pmatrix}\right) 
 =\zeta\left( \underset{j=2}{\overset{n+1}{\prod}}b_{1j}\right)\left(a_1\cdot \theta(f) \begin{pmatrix}
 a_2 & b_{23}  & b_{24} &  ...&   b_{2,n} & b_{2,n+1} \\
 1 & a_3  & b_{34} &  ...&   b_{3,n} & b_{3,n+1} \\
 . & .  & . &  ...&  . & .\\
 1 & 1 & 1 &  ...  & a_{n} & b_{n,n+1}\\
 1 & 1 & 1 &  ...  & 1 & a_{n+1} \\
\end{pmatrix}\right)(c)
\end{equation*}
 Looking at the other terms on the right hand side of \eqref{maindiff}, it is now evident that
\begin{equation*}
\tiny
\delta^n(f)\left(c\bigotimes \begin{pmatrix}
a_1 & b_{12} & b_{13} & b_{14} & ...& b_{1,n-1} & b_{1,n} & b_{1,n+1} \\
1 & a_2 & b_{23}  & b_{24} &  ...& b_{2,n-1} & b_{2,n} & b_{2,n+1} \\
1 & 1 & a_3  & b_{34} &  ...& b_{3,n-1} & b_{3,n} & b_{3,n+1} \\
. & . & .  & . &  ...& . & . & .\\
1 & 1 & 1 & 1 &  ...& 1 & a_{n} & b_{n,n+1}\\
1 & 1 & 1 & 1 &  ...& 1 & 1 & a_{n+1}\\
\end{pmatrix}\right)= \delta'(\theta(f))\begin{pmatrix}
a_1 & b_{12} & b_{13} & b_{14} & ...& b_{1,n-1} & b_{1,n} & b_{1,n+1} \\
1 & a_2 & b_{23}  & b_{24} &  ...& b_{2,n-1} & b_{2,n} & b_{2,n+1} \\
1 & 1 & a_3  & b_{34} &  ...& b_{3,n-1} & b_{3,n} & b_{3,n+1} \\
. & . & .  & . &  ...& . & . & .\\
1 & 1 & 1 & 1 &  ...& 1 & a_{n} & b_{n,n+1}\\
1 & 1 & 1 & 1 &  ...& 1 & 1 & a_{n+1}\\
\end{pmatrix}(c)\\
\end{equation*} This proves the result.
\end{proof}

We conclude this section by showing that the secondary Hochschild complex of an entwining structure contains the secondary Hochschild complex of Staic \cite{St1}.

\begin{thm}\label{P2.4} There is a canonical inclusion of complexes
\begin{equation}\label{eq2.11}
j: \mathscr C^\bullet((A,B,\zeta);M)\hookrightarrow \mathscr C^\bullet_\psi((A,B,C,\zeta);M)\qquad f\mapsto f\circ (\varepsilon \otimes id)
\end{equation} where $\varepsilon : C\longrightarrow k$ is the counit morphism of the coalgebra $C$.

\end{thm}

\begin{proof} We have assumed that $C\ne 0$ and hence $\varepsilon : C\longrightarrow k$ is an epimorphism. Then, the morphisms $\varepsilon \otimes id : C\otimes A^{\otimes n}\otimes B^{\otimes \frac{n(n-1)}{2}}\longrightarrow A^{\otimes n}\otimes B^{\otimes \frac{n(n-1)}{2}}$  are all epimorphisms and hence they induce a monomorphism $j:\mathscr C^\bullet((A,B,\zeta);M)\hookrightarrow \mathscr C^\bullet_\psi((A,B,C,\zeta);M)$. Using the fact that 
$a_\psi\varepsilon (c^\psi)=\varepsilon (c)a$ for all $a\in A$, $c\in C$, it may be verified easily that $\delta (f\circ (\varepsilon \otimes id))=\delta(f)\circ (\varepsilon \otimes id)$ for 
each $f\in \mathscr C^\bullet((A,B,\zeta);M)$.
\end{proof}

\section{Cup products on the secondary complex of an entwining structure}

\medskip
In this section, we will always assume that $M=A$. In that case, the complex $\mathscr C^\bullet_\psi((A,B,C,\zeta);A)$ will be written simply
as $\mathscr C^\bullet_\psi(A,B,C,\zeta)$. Our purpose is to produce cup products on the complex $\mathscr C^\bullet_\psi(A,B,C,\zeta)$ as well as the induced
products on the Hochschild cohomology $HH^\bullet_\psi(A,B,C,\zeta)$. 

\smallskip
In order to proceed further, we will need some simplifying notation. 

\smallskip
(A) For an element $M=((m_{ij}))_{1\leq i,j\leq n}\in A^{\otimes n}\otimes B^{\otimes \frac{n(n-1)}{2}}$ as in \eqref{tensmat} we denote by $M((k,l);(k',l'))$
the $(k'-k+1)\times (l'-l+1)$-tensor matrix consisting of those entries $m_{ij}$ such that $k\leq i\leq k'$ and $l\leq j\leq l'$. 

\medskip
(B) Given a $k\times l$-tensor matrix $T=((t_{ij}))\in B^{\otimes kl}$:

\begin{enumerate}
\item We let $\Pi^rT$ be the $(k\times 1)$-column matrix in $B^{\otimes k}$ whose $i$-th element is given by $\prod_{j=1}^{l}t_{ij}$. 

\item We let $\Pi^cT$ be the $(1\times l)$-row matrix in $B^{\otimes l}$ whose $j$-th element is given by $\prod_{i=1}^{k}t_{ij}$. 

\item We let $\Pi T\in B$ be the single element obtained by multiplying together all entries in $T$, i.e., $\Pi T=\Pi^r(\Pi^cT)=\Pi^c(\Pi^rT)$. 
\end{enumerate} For any $k\leq n$, the map $C\otimes A^{\otimes n}\otimes B^{\otimes \frac{n(n-1)}{2}}
\longrightarrow A^{\otimes k}\otimes B^{\otimes \frac{k(k-1)}{2}}\otimes B^{\otimes k(n-k)}\otimes C\otimes A^{\otimes (n-k)}\otimes B^{\otimes \frac{(n-k)(n-k-1)}{2}}$ induced  by applying $k$-times the entwining $\psi: C\otimes A\longrightarrow A\otimes C$ will be written as
\begin{equation*}
c\otimes \begin{pmatrix}
M((1,1);(k,k)) & M((1,k+1);(k,n))\\
M((k+1,1);(n,k)) & M((k+1,k+1);(n,n)) \\
\end{pmatrix}\mapsto \begin{pmatrix}
M((1,1);(k,k))_\psi & M((1,k+1);(k,n))\\
M((k+1,1);(n,k)) & c^{\psi^k}\otimes M((k+1,k+1);(n,n)) \\
\end{pmatrix}
\end{equation*} We will now describe two different cup products on $\mathscr C^\bullet_\psi(A,B,C,\zeta)$. The first is given by taking for any $f\in \mathscr C^m_\psi(A,B,C,\zeta)$, 
$g\in \mathscr C^n_\psi(A,B,C,\zeta)$, the element $f\cup g\in \mathscr C^{m+n}_\psi(A,B,C,\zeta)$  described by setting
\begin{equation}\label{cup1}
\begin{array}{l}
(f\cup g)(c\otimes M)\\ =\zeta\left(\Pi M((1,m+1);(m,m+n))\right)\cdot f(c_1\otimes M((1,1);(m,m))_\psi)\cdot g(c_2^{\psi^m}\otimes M((m+1,m+1);(n,n)))\\
\end{array}
\end{equation} for any $c\otimes M\in C\otimes A^{\otimes (m+n)}\otimes B^{\otimes \frac{(m+n)(m+n-1)}{2}}$. The second cup product  $f\sqcup g\in \mathscr C^{m+n}_\psi(A,B,C,\zeta)$ is described by setting
\begin{equation}\label{cup2}
\begin{array}{l}
(f\sqcup g)(c\otimes M)\\ =\zeta\left(\Pi M((1,m+1);(m,m+n))\right)\cdot f(c_2\otimes M((1,1);(m,m)))_\psi \cdot g(c_1^\psi \otimes M((m+1,m+1);(n,n)))\\
\end{array}
\end{equation} In order to study the cup products in \eqref{cup1} and \eqref{cup2}, we will show more generally that $\mathscr C_\psi(A,B,C,\zeta)$ is a ``weak comp algebra'' in
the sense of \cite[Definition 4.4]{Brz2}. This notion was introduced by Brzezi\'{n}ski \cite{Brz2} as a generalization of the notion of a comp algebra introduced in  \cite{G1}, \cite{GS2}. 

\begin{defn} \label{D3.1} (see \cite[Definition 4.4]{Brz2})
Let $k$ be a field. A (right) weak comp algebra $(V^\bullet,\Diamond,\alpha)$ over $k$ consists of the following data:

\smallskip
(A) A graded vector space $V=\underset{i\geq 0}{\bigoplus}V^i$ and a given element $\alpha\in V^2$

\smallskip
(B) A family $\Diamond$ of  operations 
\begin{equation}
\Diamond_i:V^m\otimes V^n\longrightarrow V^{m+n-1} \qquad \forall\textrm{ }i\geq 0
\end{equation} satisfying the following conditions for any $f\in V^m$, $g\in V^n$, $h\in V^p$:

\smallskip
\begin{enumerate}
\item $f\Diamond_ig =0$ if $i>m-1$.

\item $(f\Diamond_i g)\Diamond_j h=f\Diamond_i(g\Diamond_{j-i}h)$ if $i\leq j<n+i$

\item if either $g=\alpha$ or $h=\alpha$, then $(f\Diamond_i g)\Diamond_j h=(f\Diamond_j h)\Diamond_{i+p-1}g$ if $j<i$

\item $\alpha\Diamond_0\alpha = \alpha\Diamond_1\alpha$.
\end{enumerate}
\end{defn}

For $f\in \mathscr C^m_\psi(A,B,C,\zeta)$, 
$g\in \mathscr C^n_\psi(A,B,C,\zeta)$ and $0\leq i<m$, we now define $f\Diamond_ig\in \mathscr C_\psi^{m+n-1}(A,B,C,\zeta)$ by setting
\begin{equation*}\small
\begin{array}{l}
(f\Diamond_ig)(c\otimes M)\\
=f\left(c_1\otimes \begin{pmatrix} 
M((1,1);(i,i))_\psi & \Pi^rM((1,i+1);(i,i+n))& M((1,i+n+1);(i,m+n-1))\\
1 & g(c_2^{\psi^i}\otimes M((i+1,i+1);(i+n,i+n))) &\Pi^c M((i+1,i+n+1);(i+n,m+n-1))\\
1 & 1 & M((i+n+1,i+n+1);(m+n-1,m+n-1)) \\
\end{pmatrix}\right)\\
\end{array}
\end{equation*} For all other values of $i$, we take $f\Diamond_ig=0$. We also fix an element $\alpha\in \mathscr C_\psi^2(A,B,C,\zeta)$ defined by setting 
\begin{equation}\label{alpha}
\alpha\left(c\otimes \begin{pmatrix} a_1 & b_{12} \\ 1 & a_2 \\ \end{pmatrix}\right):=\varepsilon(c)\zeta(b_{12})a_1a_2
\end{equation} 

\begin{lem}\label{L3.2} For $f\in \mathscr C^m_\psi(A,B,C,\zeta)$, 
$g\in \mathscr C^n_\psi(A,B,C,\zeta)$, $h\in \mathscr C_\psi^p(A,B,C,\zeta)$, we have 
\begin{equation}\label{3.5eq}
 (f\Diamond_i g)\Diamond_j h=f\Diamond_i(g\Diamond_{j-i}h) \in \mathscr C_\psi^{m+n+p-2}(A,B,C,\zeta)
 \end{equation} for $i\leq j< n+i$. 

\end{lem}

\begin{proof}
For $i\geq m$, both sides of \eqref{3.5eq} are $0$. So we take $i<m$.  We notice that this implies $j<n+i\leq m+n-1$. We now consider a dimension 
$m+n+p-2$  square  ``upper triangular tensor matrix'' $M$ (i.e., entries below the diagonal are all $1$) in $A^{\otimes (m+n+p-2)}\otimes 
B^{\otimes \frac{(m+n+p-2)(m+n+p-3)}{2}}$. For the sake of convenience, we will write $M$ as a $5\times 5$ matrix by subdividing
it into the following blocks
\begin{equation}\label{3.6eq}
M=\begin{pmatrix}
U(i) & X^{12} & X^{13} & X^{14} & X^{15} \\
1 & U(j-i) & X^{23} & X^{24} & X^{25} \\
1 & 1 & U(p) & X^{34} & X^{35} \\
1& 1 & 1 & U(n-j+i-1) & X^{45} \\
1 & 1 & 1 & 1 & U(m-i-1) \\
\end{pmatrix}
\end{equation} where each $U(k)$ in \eqref{3.6eq} is a square block of dimension $k$.  For $c\in C$, we now have
\begin{equation*}
\begin{array}{l}
f\Diamond_i(g\Diamond_{j-i}h)(c\otimes M) \\
 = f\left(c_1\otimes \left(
 \begin{matrix}
 U(i)_\psi & \Pi^r(X^{12}; X^{13} ; X^{14}) &  X^{15} \\ 
 1 & (g\Diamond_{j-i}h) \left(c_2^{\psi^i}\otimes \left(
 \begin{matrix}
 U(j-i) & X^{23} & X^{24}\\
 1 & U(p) & X^{34}\\ 
 1 & 1 &  U(n-j+i-1) \\
 \end{matrix} \right) \right)& \Pi^c(X^{25}; X^{35} ; X^{45})   \\
 1 & 1 & U(m-i-1)\\ 
 \end{matrix} \right) \right)\\  \\
 = f\left(c_1\otimes \left(
 \begin{matrix}
 U(i)_\psi & \Pi^r(X^{12}; X^{13} ; X^{14}) &  X^{15} \\ 
 1 & g\left((c_2^{\psi^i})_1\otimes \left(
 \begin{matrix}
 U(j-i)_\psi & \Pi^r X^{23} & X^{24}\\
 1 & h\left((c_2^{\psi^i})_2^{\psi^{j-i}}\otimes U(p)\right) & \Pi^c X^{34}\\ 
 1 & 1 &  U(n-j+i-1) \\
 \end{matrix} \right) \right)& \Pi^c(X^{25}; X^{35} ; X^{45})   \\
 1 & 1 & U(m-i-1)\\ 
 \end{matrix} \right) \right)\\ 
\end{array}
\end{equation*}  On the other hand, we have
\begin{equation*}
\begin{array}{l}
( (f\Diamond_i g)\Diamond_j h)(c\otimes M)\\
= (f\Diamond_ig)\left(c_1\otimes \left(\begin{matrix}
U(i)_\psi & X^{12} & \Pi^r X^{13} & X^{14} & X^{15} \\
1 & U(j-i)_\psi & \Pi^r X^{23} & X^{24} & X^{25} \\
1 & 1 & h(c_2^{\psi^j}\otimes U(p)) & \Pi^c X^{34} & \Pi^c X^{35} \\
1& 1 & 1 & U(n-j+i-1) & X^{45} \\
1 & 1 & 1 & 1 & U(m-i-1) \\
\end{matrix} \right) \right)\\ 
= f\left(c_{1}\otimes \left(\begin{matrix}
U(i)_{\psi\psi} & \Pi^r (X^{12}; X^{13}; X^{14}) & X^{15}   \\
1 & g\left(c_{2}^{\psi^i}\otimes \left(
\begin{matrix}
U(j-i)_\psi & \Pi^r X^{23} & X^{24}\\ 
1 & h(c_3^{\psi^j}\otimes U(p)) & \Pi^c X^{34} \\
1 & 1 & U(n-j+i-1) \\
\end{matrix} \right)\right) & \Pi^c(X^{25}; X^{35} ; X^{45})   \\ 
1 & 1 & U(m-i-1) \\
\end{matrix} \right) \right)
\end{array}
\end{equation*} From \eqref{ent}, we know that $a_\psi\otimes \Delta({d}^\psi)={a_{\psi}}_{\psi}\otimes {d}_1^\psi \otimes {d}_2^\psi $  
for
$a\in A$ and $d\in C$. Applying the $i$-th iterate of this condition, the result is now clear. 
\end{proof}

Given an entwining structure $(A,C,\psi)$, the conditions in Definition \ref{D2.1} imply that $\psi$ interacts with the multiplication $\mu: A\otimes A
\longrightarrow A$ and the comultiplication $\Delta: C\longrightarrow C\otimes C$ in the following manner: for any $n\geq 1$, set
\begin{equation}\label{coac}
\rho^n_R: C\otimes A^{\otimes n} \xrightarrow{\textrm{ }\Delta\otimes A^{\otimes n}} C\otimes C\otimes A^{\otimes n}
\xrightarrow{C\otimes ((A^{\otimes n-1}\otimes \psi)\circ (A^{\otimes n-2}\otimes \psi \otimes A)\circ \dots \circ (\psi\otimes A^{\otimes n-1}))} C\otimes A^{\otimes n}\otimes C
\end{equation} Then, for any $0\leq j\leq n-1$, the following diagram commutes (see \cite[Lemma 1.3]{Brz2})
\begin{equation}\label{lemeq}
\begin{CD}
C\otimes A^{\otimes n+1} @>C\otimes A^{\otimes j}\otimes \mu\otimes A^{\otimes n-j-1} >>  C\otimes A^{\otimes n}\\
@V\rho^{n+1}_R VV @V\rho^{n}_RVV \\
C\otimes A^{\otimes n+1}\otimes C@>C\otimes A^{\otimes j}\otimes \mu\otimes A^{\otimes n-j-1}\otimes C>> C\otimes A^{\otimes n}\otimes C
\end{CD}
\end{equation}

\begin{lem}\label{L3.3} For $f\in \mathscr C^m_\psi(A,B,C,\zeta)$, 
$g\in \mathscr C^n_\psi(A,B,C,\zeta)$, $h\in \mathscr C_\psi^p(A,B,C,\zeta)$ and $j<i$, we have 
\begin{equation}\label{3.5eqx}
(f\Diamond_i g)\Diamond_j h=(f\Diamond_j h)\Diamond_{i+p-1}g
 \end{equation} whenever $g=\alpha$ or $h=\alpha$.  

\end{lem}

\begin{proof}
Suppose first that $h=\alpha$. Then, $p=2$ and we have to establish that
\begin{equation}\label{3.w} (f\Diamond_i g)\Diamond_j\alpha=(f\Diamond_j \alpha)\Diamond_{i+1}g
\end{equation} for $j<i$. We now consider a dimension 
$m+n$  square  ``upper triangular tensor matrix'' $M$ (i.e., entries below the diagonal are all $1$) in $A^{\otimes (m+n)}\otimes 
B^{\otimes \frac{(m+n)(m+n-1)}{2}}$. For the sake of convenience, we will write $M$ as a $6\times 6$ matrix by subdividing
it into the following blocks
\begin{equation}\label{3.10eq}
M=\begin{pmatrix}
U(j) & X^{12} & X^{13} & X^{14} & X^{15}  & X^{16} \\
1 & a_{j+1} & b  & X^{24} & X^{25} &  X^{26} \\
1 & 1 & a_{j+2} & X^{34} & X^{35}  & X^{36} \\
1& 1 & 1 & U(i-j-1) & X^{45} & X^{46} \\
1 & 1 & 1 & 1 & U(n) & X^{56} \\
1 & 1 & 1 & 1 & 1& U(m-i-1) \\
\end{pmatrix}
\end{equation}

where each $U(k)$ is a square matrix of dimension $k$. For $c\in C$, we now have
\begin{equation*}
\begin{array}{l}
((f\Diamond_i g)\Diamond_j\alpha)(c\otimes M)\\ \\ 
= (f\Diamond_i g)\left(c_1\otimes \left(
\begin{matrix}
U(j)_\psi & \Pi^r(X^{12}; X^{13}) & X^{14} & X^{15} & X^{16}  \\
1 & \varepsilon (c_2^{\psi^j}) \zeta (b) a_{j+1}a_{j+2} & \Pi^c(X^{24}; X^{34}) & \Pi^c(X^{25};X^{35}) &\Pi^c(X^{26};X^{36}) \\
1 & 1 & U(i-j-1) & X^{45} & X^{46} \\ 
1 & 1 & 1 & U(n) & X^{56} \\
1 & 1 & 1 & 1 & U(m-i-1) \\
\end{matrix} \right)\right)\\
= (f\Diamond_i g)\left(c\otimes \left(
\begin{matrix}
U(j) & \Pi^r(X^{12}; X^{13}) & X^{14} & X^{15} & X^{16}  \\
1 &  \zeta (b) a_{j+1}a_{j+2} & \Pi^c(X^{24}; X^{34}) & \Pi^c(X^{25};X^{35}) &\Pi^c(X^{26};X^{36}) \\
1 & 1 & U(i-j-1) & X^{45} & X^{46} \\ 
1 & 1 & 1 & U(n) & X^{56} \\
1 & 1 & 1 & 1 & U(m-i-1) \\
\end{matrix} \right)\right)\\ \\ 
= f\left(c_{1}\otimes \left(
\begin{matrix}
U(j)_{\psi} & \Pi^r(X^{12}; X^{13}) & X^{14} & \Pi^r X^{15} & X^{16}  \\
1 &  \zeta (b) (a_{j+1}a_{j+2})_\psi &\Pi^c(X^{24}; X^{34}) &  \Pi^r\Pi^c(X^{25};X^{35}) &\Pi^c(X^{26};X^{36}) \\
1 & 1 & U(i-j-1)_\psi & \Pi^r X^{45} & X^{46} \\ 
1 & 1 & 1 & g(c_{2}^{\psi^i}\otimes U(n)) & \Pi^c X^{56} \\
1 & 1 & 1 & 1 & U(m-i-1) \\
\end{matrix} \right)\right)
\end{array}
\end{equation*} On the other hand, we have
\begin{equation*}
\begin{array}{l} 
((f\Diamond_j \alpha)\Diamond_{i+1}g)(c\otimes M)\\ \\
=(f\Diamond_j\alpha)\left(c_1\otimes \left(
\begin{matrix}
U(j)_\psi & X^{12} & X^{13} & X^{14} & \Pi^r X^{15}  & X^{16} \\
1 & a_{j+1,\psi} & b  & X^{24} & \Pi^r X^{25}  & X^{26} \\
1 & 1 & a_{j+2,\psi} & X^{34} & \Pi^r X^{35}  & X^{36} \\
1& 1 & 1 & U(i-j-1)_\psi & \Pi^r X^{45} & X^{46} \\
1 & 1 & 1 & 1 & g(c_2^{\psi^{i+1}}\otimes U(n) ) & \Pi^c X^{56} \\
1 & 1 & 1 & 1 & 1& U(m-i-1) \\
\end{matrix} \right)\right) 
\end{array}
\end{equation*}
This further equates to
\begin{equation*}
\begin{array}{l} 
=f\left(c_{11}\otimes \left(\begin{matrix}
U(j)_{\psi\psi} & \Pi^r(X^{12};X^{13}) & X^{14} & \Pi^r X^{15}  & X^{16} \\
1 & \varepsilon(c_{12}^{\psi^j})\zeta(b) a_{j+1,\psi}a_{j+2,\psi} & \Pi^c(X^{24}; X^{34})  &   \Pi^r\Pi^c(X^{25};X^{35}) &\Pi^c(X^{26};X^{36}) \\
 1 & 1 & U(i-j-1)_\psi & \Pi^r X^{45} & X^{46} \\
 1 & 1 & 1 & g(c_2^{\psi^{i+1}}\otimes U(n) ) & \Pi^c X^{56} \\
1 & 1 & 1 & 1& U(m-i-1) \\
\end{matrix}\right) \right) \\ \\
=f\left(c_{1}\otimes \left(\begin{matrix}
U(j)_{\psi} & \Pi^r(X^{12};X^{13}) & X^{14} & \Pi^r X^{15}  & X^{16} \\
1 & \zeta(b) a_{j+1,\psi}a_{j+2,\psi} & \Pi^c(X^{24}; X^{34})  &   \Pi^r\Pi^c(X^{25};X^{35}) &\Pi^c(X^{26};X^{36}) \\
 1 & 1 & U(i-j-1)_\psi & \Pi^r X^{45} & X^{46} \\
 1 & 1 & 1 & g(c_2^{\psi^{i+1}}\otimes U(n) ) & \Pi^c X^{56} \\
1 & 1 & 1 & 1& U(m-i-1) \\
\end{matrix}\right) \right) 
\end{array}
\end{equation*} The result of \eqref{3.w} now follows using the commutative diagram \eqref{lemeq}. The case of $g=\alpha$ may be verified
in a similar manner. 
\end{proof}

\begin{Thm}\label{Th3} Let $B$ be a commutative $k$-algebra and let $(A,B,C,\psi,\zeta)$ be an entwining structure over $B$. Then, the tuple 
$(\mathscr C_\psi^\bullet(A,B,C,\zeta),\Diamond,\alpha)$ is the data of a weak comp algebra.
\end{Thm}

\begin{proof}
The condition (1) in Definition \ref{D3.1} is immediate from the definition of the operations $\Diamond_i$ on $\mathscr C_\psi^\bullet(A,B,C,\zeta)$. The conditions (2) and (3) follow from Lemma \ref{L3.2} and Lemma \ref{L3.3} respectively. Finally, the condition (4), i.e., $\alpha\Diamond_0
\alpha=\alpha\Diamond_1\alpha$, may be verified easily by direct computation. The result follows. 
\end{proof}

\begin{lem}\label{L3.5} For any $m\geq 0$, the differential $\delta^m: \mathscr C_\psi^m(A,B,C,\zeta)\longrightarrow \mathscr C_\psi^{m+1}(A,B,C,\zeta)$ can also be expressed as
\begin{equation}\delta^m(f)=(-1)^{m-1}\alpha\Diamond_0f-\sum_{i=1}^{m}(-1)^{i-1} f\Diamond_{i-1}\alpha + \alpha\Diamond_1f
\end{equation}

\end{lem}

\begin{proof}
As in \eqref{tensmat}, we will write an element of $C\otimes A^{\otimes m+1}\otimes B^{\otimes \frac{(m+1)m}{2}}$ as a ``tensor matrix''
\begin{equation}\label{tensmatw}
c\otimes M=c\bigotimes \begin{pmatrix}
a_1 & b_{12} & b_{13} & b_{14} & ...& b_{1,m-1} & b_{1,m} & b_{1,m+1} \\
1 & a_2 & b_{23}  & b_{24} &  ...& b_{2,m-1} & b_{2,m} & b_{2,m+1} \\
1 & 1 & a_3  & b_{34} &  ...& b_{3,m-1} & b_{3,m} & b_{3,m+1} \\
. & . & .  & . &  ...& . & . & .\\
1 & 1 & 1 & 1 &  ...& 1 & a_{m} & b_{m,m+1}\\
1 & 1 & 1 & 1 &  ...& 1 & 1 & a_{m+1} \\
\end{pmatrix}
\end{equation} with $c\in C$, $a_i\in A$, $b_{ij}\in B$ and $1\in k$. It is easily observed that 
\begin{equation}
\begin{array}{ll}
(\alpha\Diamond_0 f)(c\otimes M) &=\zeta\left(\underset{k=1}{\overset{m}{\prod}}b_{k,m+1}\right)f\left(c\bigotimes \begin{pmatrix}
a_1 & b_{12} & b_{13} & b_{14} & ...& b_{1,m-1} & b_{1,m}  \\
1 & a_2 & b_{23}  & b_{24} &  ...& b_{2,m-1} & b_{2,m}   \\
1 & 1 & a_3  & b_{34} &  ...& b_{3,m-1} & b_{3,m}   \\
. & . & .  & . &  ...& . & . \\
1 & 1 & 1 & 1 &  ...& 1 & a_{m}  \\
\end{pmatrix}\right)\cdot a_{m+1}
\end{array}
\end{equation}
\begin{equation}
\begin{array}{ll}
(\alpha\Diamond_1 f)(c\otimes M) & = \varepsilon(c_1)\zeta\left( \underset{j=2}{\overset{m+1}{\prod}}b_{1j}\right) a_{1\psi}\cdot f\left(c_2^\psi\bigotimes \begin{pmatrix}
 a_2 & b_{23}  & b_{24} &  ...& b_{2,m-1} & b_{2,m} & b_{2,m+1} \\
 1 & a_3  & b_{34} &  ...& b_{3,m-1} & b_{3,m} & b_{3,m+1} \\
 . & .  & . &  ...& . & . & .\\
 1 & 1 & 1 &  ...& 1 & a_{m} & b_{m,m+1}\\
 1 & 1 & 1 &  ...& 1 & 1 & a_{m+1} \\
\end{pmatrix}\right)\\ \\ 
 & = \zeta\left( \underset{j=2}{\overset{m+1}{\prod}}b_{1j}\right) a_{1\psi}\cdot f\left(c^\psi\bigotimes \begin{pmatrix}
 a_2 & b_{23}  & b_{24} &  ...& b_{2,m-1} & b_{2,m} & b_{2,m+1} \\
 1 & a_3  & b_{34} &  ...& b_{3,m-1} & b_{3,m} & b_{3,m+1} \\
 . & .  & . &  ...& . & . & .\\
 1 & 1 & 1 &  ...& 1 & a_{m} & b_{m,m+1}\\
 1 & 1 & 1 &  ...& 1 & 1 & a_{m+1}
\end{pmatrix}\right)
\end{array}
\end{equation} On the other hand, we have
\begin{equation}\label{etns}
\begin{array}{ll}
(f\Diamond_{i-1}\alpha)(c\otimes M) &= f\left(c_1\bigotimes \begin{pmatrix}
a_{1\psi} & b_{12} & \dots & b_{1i}b_{1,i+1} & \dots & b_{1,m+1}\\
1 & a_{2\psi}   & \dots & b_{2i}b_{2,i+1} & \dots & b_{2,m+1} \\
. & .   & \dots &\dots & \dots & . \\
1 & 1  & \dots & \varepsilon(c_2^{\psi^{i-1}})\zeta(b_{i,i+1})a_{i}a_{i+1} &  \dots & b_{i,m+1}b_{i+1,m+1}\\
. & .  & \dots & \dots & \dots & . \\
1 & 1  & \dots &\dots &  \dots & b_{m,m+1}\\
1 & 1  & \dots & \dots &\dots & a_{m+1} 
\end{pmatrix}\right) 
\end{array}
\end{equation}
\begin{equation*}
\begin{array}{ll}
= f\left(c_1\bigotimes \begin{pmatrix}
a_{1} & b_{12} & \dots & b_{1i}b_{1,i+1} & \dots & b_{1,m+1}\\
1 & a_{2}   & \dots & b_{2i}b_{2,i+1} & \dots & b_{2,m+1} \\
. & .   & \dots &\dots & \dots & . \\
1 & 1  & \dots & \varepsilon(c_2)\zeta(b_{i,i+1})a_{i}a_{i+1} &  \dots & b_{i,m+1}b_{i+1,m+1}\\
. & .  & \dots & \dots & \dots & . \\
1 & 1  & \dots &\dots &  \dots & b_{m,m+1}\\
1 & 1  & \dots & \dots &\dots & a_{m+1} 
\end{pmatrix}\right) 
\end{array}
\end{equation*}
\begin{equation*}
\begin{array}{ll}
&= f\left(c\bigotimes \begin{pmatrix}
a_{1} & b_{12} & \dots & b_{1i}b_{1,i+1} & \dots & b_{1,m+1}\\
1 & a_{2}   & \dots & b_{2i}b_{2,i+1} & \dots & b_{2,m+1} \\
. & .   & \dots &\dots & \dots & . \\
1 & 1  & \dots &\zeta(b_{i,i+1})a_{i}a_{i+1} &  \dots & b_{i,m+1}b_{i+1,m+1}\\
. & .  & \dots & \dots & \dots & . \\
1 & 1  & \dots &\dots &  \dots & b_{m,m+1}\\
1 & 1  & \dots & \dots &\dots & a_{m+1} \\
\end{pmatrix}\right) \\ 
\end{array}
\end{equation*} Comparing with the expression for the differential $\delta^m: \mathscr C_\psi^m(A,B,C,\zeta)\longrightarrow \mathscr C_\psi^{m+1}(A,B,C,\zeta)$ in \eqref{maindiff}, the result is clear. 
\end{proof}

We now come back to the two cup products $\cup$ and $\sqcup$ on $\mathscr C_\psi^\bullet(A,B,C,\zeta)$ defined in \eqref{cup1} and \eqref{cup2} and describe them in terms of the operations $\Diamond_i$. This will allow us to produce associative algebra structures on  $\mathscr C_\psi^\bullet(A,B,C,\zeta)$ equipped with graded differential $\delta$ by using the general properties of weak comp algebras. 

\begin{thm}\label{P3.6}  Let $B$ be a commutative $k$-algebra and let $(A,B,C,\psi,\zeta)$ be an entwining structure over $B$.  Then, the operation
$\cup$  on $\mathscr C_\psi^\bullet(A,B,C,\zeta)$ may also be described as follows: for $f\in \mathscr C_\psi^m(A,B,C,\zeta)$, $g\in 
\mathscr C_\psi^n(A,B,C,\zeta)$, we have
\begin{equation}\label{cupcom}
f\cup g=(\alpha\Diamond_0 f)\Diamond_mg 
\end{equation}
Further, $\mathscr C^\bullet_\psi(A,B,C,\zeta)$ is a graded associative algebra for the $\cup$ product. The differential $\delta$ is a graded derivation
of degree $1$ for this algebra structure, i.e., 
\begin{equation}\label{diffcup}
\delta^{m+n}(f\cup g)=\delta^m(f)\cup g + (-1)^mf\cup \delta^n(g)
\end{equation}
\end{thm} 

\begin{proof}
We write a square tensor matrix of dimension $(m+n)$ as a $2\times 2$ matrix with the following blocks
\begin{equation}
M=\begin{pmatrix} U(m) & X^{12}  \\ 1 & U(n) \\
\end{pmatrix}
\end{equation} where each $U(k)$ is a square matrix of dimension $k$. For $c\in C$, we have
\begin{equation*}
\begin{array}{ll}
&((\alpha\Diamond_0 f)\Diamond_mg)(c\otimes M) =\alpha\left(c_{11}\otimes 
\left(
\begin{pmatrix}
f(c_{12}\otimes U(m)_\psi)& \Pi^c\Pi^rX^{12} \\
1 & g(c_2^{\psi^m}\otimes U(n))\\
\end{pmatrix}\right)\right)\\ 
&=\zeta\left(\Pi^c\Pi^rX^{12}\right)\varepsilon(c_{11})f(c_{12}\otimes U(m)_\psi) g(c_2^{\psi^m}\otimes U(n))=\zeta\left(\Pi^c\Pi^rX^{12}\right)f(c_{1}\otimes U(m)_\psi) g(c_2^{\psi^m}\otimes U(n))
\end{array}
\end{equation*} It is clear that this coincides with the expression for $(f\cup g)(c\otimes M)$ as defined in \eqref{cup1}. This proves \eqref{cupcom}. From Theorem \ref{Th3}, we know that $(\mathscr C^\bullet_\psi(A,B,C,\zeta),\Diamond,\alpha)$ is a weak comp algebra. It now follows from the general properties of a weak comp algebra in \cite[Proposition 4.5]{Brz2} that $\cup$ defines a graded associative algebra structure on $\mathscr C^\bullet_\psi(A,B,C,\zeta)$. Using the expression for the differential $\delta$ given in Lemma \ref{L3.5}, it also follows from the general properties of a weak
comp algebra (see \cite[Proposition 4.7]{Brz2}) that $\delta$ is a graded derivation of degree $1$ for the cup product $\cup$. 
\end{proof}

\begin{thm}\label{P3.7}  Let $B$ be a commutative $k$-algebra and let $(A,B,C,\psi,\zeta)$ be an entwining structure over $B$.  Then, the operation
$\sqcup$  on $\mathscr C_\psi^\bullet(A,B,C,\zeta)$ may also be described as follows: for $f\in \mathscr C_\psi^m(A,B,C,\zeta)$, $g\in 
\mathscr C_\psi^n(A,B,C,\zeta)$, we have
\begin{equation}\label{sqcupcom}
f\sqcup g=(\alpha\Diamond_1 g)\Diamond_0f
\end{equation}
Further, $\mathscr C^\bullet_\psi(A,B,C,\zeta)$ is an associative algebra with the product $\sqcup$. The differential $\delta$ is a graded derivation
of degree $1$ for this algebra structure, i.e., 
\begin{equation}\label{diffsqcup}
\delta^{m+n}(f\sqcup g)=\delta^m(f)\sqcup g + (-1)^mf\sqcup \delta^n(g)
\end{equation}
\end{thm} 

\begin{proof}
We maintain the notation from the proof of Proposition \ref{P3.6}. For $c\in C$, we have
\begin{equation*}
\begin{array}{ll}
((\alpha\Diamond_1 g)\Diamond_0f)(c\otimes M)& = (\alpha\Diamond_1 g)\left(
c_1\otimes  
\begin{pmatrix}
f(c_2\otimes U(m)) & \Pi^c X^{12} \\ 
1 & U(n) \\
\end{pmatrix}
\right)= \alpha\left(c_{11}\otimes  
\begin{pmatrix}
f(c_2\otimes U(m))_\psi & \Pi^r \Pi^c X^{12} \\ 
1 & g(c_{12}^\psi\otimes U(n)) \\
\end{pmatrix} \right)\\ \\
& = \zeta\left(\Pi^r \Pi^c X^{12}\right)\varepsilon(c_{11})f(c_2\otimes U(m))_\psi g(c_{12}^\psi\otimes U(n))= \zeta\left(\Pi^r \Pi^c X^{12}\right)f(c_2\otimes U(m))_\psi g(c_{1}^\psi\otimes U(n)) 
\end{array}
\end{equation*}  It is clear that this coincides with the expression for $(f\sqcup g)(c\otimes M)$ as defined in \eqref{cup2}. This proves \eqref{sqcupcom}. The rest follows as in the proof of Proposition \ref{P3.6} by using the general properties of a weak comp algebra (see \cite[Proposition 4.5 \& Proposition 4.7]{Brz2}).
\end{proof}

\begin{cor}\label{C3.8}  Let $B$ be a commutative $k$-algebra and let $(A,B,C,\psi,\zeta)$ be an entwining structure over $B$. Then, 

\smallskip
(a) There are cup products on the secondary Hochschild cohomology of the entwining structure
\begin{equation*}
\begin{array}{c}
\cup : HH^m_\psi(A,B,C,\zeta)\otimes HH_\psi^n(A,B,C,\zeta) \longrightarrow HH_\psi^{m+n}(A,B,C,\zeta) \\ \sqcup: HH^m_\psi(A,B,C,\zeta)\otimes HH_\psi^n(A,B,C,\zeta) \longrightarrow HH_\psi^{m+n}(A,B,C,\zeta)
\end{array}
\end{equation*} induced by the corresponding cup product structures on $\mathscr C_\psi^\bullet(A,B,C,\zeta)$. 

\smallskip
(b) For cohomology classes $\bar{f}\in HH^m_\psi(A,B,C,\zeta)$ and $\bar{g}\in HH^n_\psi(A,B,C,\zeta)$, we have
\begin{equation*}
\bar{f}\cup \bar{g}=(-1)^{mn}\bar{g}\sqcup \bar{f} 
\end{equation*}
\end{cor}

\begin{proof} From Propositions \ref{P3.6} and \ref{P3.7}, we know that the differential $\delta$ is a graded derivation for the associative
algebra structures $(\mathscr C_\psi^\bullet(A,B,C,\zeta),\cup)$ and $(\mathscr C_\psi^\bullet(A,B,C,\zeta),\sqcup)$. The result of (a) is now clear from 
\eqref{diffcup} and \eqref{diffsqcup}. The result of (b) now follows from the general property of a weak comp algebra (see \cite[Corollary 4.9]{Brz2}).  
\end{proof}

\section{A graded commutative cup product}

\smallskip
In this section, we will consider a subcomplex of  $\mathscr C_\psi^\bullet(A,B,C,\zeta)$ such that the two cup products on cohomology in Corollary \ref{C3.8} coincide, thus determining a graded commutative structure. We begin
by recalling the morphism 
\begin{equation}\label{coac4}
\rho^n_R: C\otimes A^{\otimes n} \xrightarrow{\textrm{ }\Delta\otimes A^{\otimes n}} C\otimes C\otimes A^{\otimes n}
\xrightarrow{C\otimes ((A^{\otimes n-1}\otimes \psi)\circ (A^{\otimes n-2}\otimes \psi \otimes A)\circ \dots (\psi\otimes A^{\otimes n-1}))} C\otimes A^{\otimes n}\otimes C
\end{equation} described in Section 3. We also set
\begin{equation}\label{coac41}
\rho^n_L:   C\otimes A^{\otimes n} \xrightarrow{\textrm{ }\Delta\otimes A^{\otimes n}} C\otimes C\otimes A^{\otimes n}
\end{equation} In fact, $C\otimes A^{\otimes n}$ is a $C$-bicomodule with left coaction $\rho^n_L$ and right coaction $\rho^n_R$ (see \cite[$\S$ 1]{Brz2}).  By  abuse
of notation, we will also write
\begin{equation}\label{4.3prq}
\begin{array}{c}
\rho^n_L:C\otimes A^{\otimes n}\otimes B^{\otimes \frac{n(n-1)}{2}} \quad \xrightarrow{\rho^n_L\otimes B^{\otimes \frac{n(n-1)}{2}}}\quad C\otimes C\otimes A^{\otimes n}\otimes B^{\otimes \frac{n(n-1)}{2}} \\ \\
\rho^n_R: C\otimes A^{\otimes n}\otimes B^{\otimes \frac{n(n-1)}{2}} \quad \xrightarrow{\rho^n_R\otimes B^{\otimes \frac{n(n-1)}{2}}}\quad  C\otimes A^{\otimes n}\otimes C\otimes  B^{\otimes \frac{n(n-1)}{2}} = C\otimes A^{\otimes n}\otimes  B^{\otimes \frac{n(n-1)}{2}} \otimes C
\end{array}
\end{equation}

\smallskip Motivated by
\cite[$\S$ 5]{Brz2}, we now introduce an ``equivariant subcomplex''   $\mathscr E_\psi^\bullet(A,B,C,\zeta)$ of  $\mathscr C_\psi^\bullet(A,B,C,\zeta)$ as follows: an element
$f\in \mathscr C^n_\psi(A,B,C,\zeta)=Hom_k(C\otimes A^{\otimes n}\otimes B^{\otimes \frac{n(n-1)}{2}},A)$ lies in $\mathscr E_\psi^n(A,B,C,\zeta)$ if it satisfies
\begin{equation}\label{eqsub}
\begin{CD}
C\otimes A^{\otimes n}\otimes B^{\otimes \frac{n(n-1)}{2}} @>\rho^n_R >> C\otimes A^{\otimes n}\otimes  B^{\otimes \frac{n(n-1)}{2}}  \otimes C \\
@V\rho_L^nVV @VVf\otimes id  V \\
C\otimes C\otimes A^{\otimes n}\otimes B^{\otimes \frac{n(n-1)}{2}}  @>\psi\circ (id\otimes  f)>> A\otimes C \\
\end{CD}
\end{equation}

\begin{lem}\label{L4.1} The element $\alpha\in \mathscr E^2_\psi(A,B,C,\zeta)$.

\end{lem}

\begin{proof} Using the definition of $\alpha$ in \eqref{alpha}, we verify the commutativity of the diagram \eqref{eqsub} as follows:
\begin{equation}
(\alpha \otimes id)\circ \rho^2_R \left(c\otimes \begin{pmatrix}a_1 & b_{12} \\ 1 & a_2 \end{pmatrix}\right)=\varepsilon(c_1)\zeta(b_{12})a_{1\psi}a_{2\psi}\otimes c_2^{\psi\psi}=\zeta(b_{12})a_{1\psi}a_{2\psi}\otimes c^{\psi\psi}
\end{equation} On the other hand, we have
\begin{equation}
(\psi\circ (id\otimes \alpha))\circ \rho^2_L\left(c\otimes \begin{pmatrix}a_1 & b_{12} \\ 1 & a_2 \end{pmatrix}\right)=\varepsilon(c_2)\zeta(b_{12})(a_1a_2)_\psi\otimes c_1^\psi=\zeta(b_{12})(a_1a_2)_\psi\otimes c^\psi
\end{equation} The result is now clear from the properties of an entwining structure as described in Definition \ref{D2.1}.
\end{proof}

\begin{lem}\label{L4.2}  If $f\in \mathscr E^m_\psi(A,B,C,\zeta)$ and $g\in \mathscr E^n_\psi(A,B,C,\zeta)$, then $f\Diamond_ig\in \mathscr E^{m+n-1}_\psi(A,B,C,\zeta)$ for any $i\geq 0$.

\end{lem}

\begin{proof} We know that $f\Diamond_ig=0$ for $i>m-1$, so we suppose that $i\leq m-1$. We now consider an ``upper triangular tensor matrix'' $M$ in $A^{\otimes m+n-1}\otimes B^{\otimes (m+n-1)(m+n-2)/2}$. As usual, we will write $M$ as a $3\times 3$ matrix with the following blocks
\begin{equation}
M=\begin{pmatrix}
U(i) & X^{12} & X^{13} \\
1 & U(n) & X^{23} \\
1 & 1 & U(m-i-1) \\
\end{pmatrix}
\end{equation} where $U(k)$ is a square block of dimension $k$. For $c\in C$, we now have
\begin{equation}\label{eq4.8}
\begin{array}{ll}
((f\Diamond_i g )\otimes id)\circ \rho^{m+n-1}_R(c\otimes M)&= f\left(c_1\otimes \begin{pmatrix}
U(i)_{\psi\psi} & \Pi^rX^{12} & X^{13}\\
1 & g(c_2^{\psi^i}\otimes U(n)_\psi )& \Pi^cX^{23} \\
1&1& U(m-i-1)_\psi\\
\end{pmatrix}\right)\otimes c_3^{\psi^{m+n-1}}\\
& =  f\left(c_1\otimes \begin{pmatrix}
U(i)_{\psi\psi} & \Pi^rX^{12} & X^{13}\\
1 & g(c_3^{\psi^i}\otimes U(n) )_\psi & \Pi^cX^{23} \\
1&1& U(m-i-1)_\psi\\
\end{pmatrix}\right)\otimes c_2^{\psi^{m}}\\
\end{array}
\end{equation} where the second equality in \eqref{eq4.8} follows from the fact that $g\in \mathscr E^n_\psi(A,B,C,\zeta)$. We also have
\begin{equation}\label{eq4.9}
\begin{array}{ll}
(\psi\circ (id\otimes (f\Diamond_ig)))\circ \rho_L^{m+n-1}(c\otimes M)&= (\psi\circ (id \otimes (f\Diamond_ig))) (c_1\otimes c_2\otimes M)= ((f\Diamond_ig)(c_2\otimes M))_\psi\otimes c_1^\psi\\
&= f\left(c_2\otimes \begin{pmatrix}
U(i)_\psi & \Pi^rX^{12} &  X^{13} \\
1 & g(c_3^{\psi^i} \otimes U(n))& \Pi^cX^{23} \\
1 & 1 & U(m-i-1) \\
\end{pmatrix} \right)_\psi\otimes c_1^\psi \\
&= f\left(c_1\otimes \begin{pmatrix}
U(i)_{\psi\psi} & \Pi^rX^{12} &  X^{13} \\
1 & g(c_3^{\psi^i} \otimes U(n))_{\psi} & \Pi^cX^{23} \\
1 & 1 & U(m-i-1)_\psi \\
\end{pmatrix} \right)\otimes c_2^{\psi^m}\\
\end{array}
\end{equation} where the last equality in \eqref{eq4.9} follows from the fact that $f\in \mathscr E^m_\psi(A,B,C,\zeta)$. This proves the result.
\end{proof}

\begin{Thm}\label{T4.3}  Let $B$ be a commutative $k$-algebra and let $(A,B,C,\psi,\zeta)$ be an entwining structure over $B$. Then:

\smallskip
(1)  The tuple 
$(\mathscr E_\psi^\bullet(A,B,C,\zeta),\Diamond,\alpha)$ is the data of a weak comp subalgebra of $(\mathscr C_\psi^\bullet(A,B,C,\zeta),\Diamond,\alpha)$.

\smallskip
(2) $(\mathscr E_\psi^\bullet(A,B,C,\zeta),\delta^\bullet)$ is a subcomplex of $(\mathscr C_\psi^\bullet(A,B,C,\zeta),\delta^\bullet)$.

\smallskip
(3) The cup products $\cup$ and $\sqcup$ on the complex  $\mathscr E_\psi^\bullet(A,B,C,\zeta)$ coincide. In particular, these induce a graded commutative cup product on the cohomology
groups of $\mathscr E_\psi^\bullet(A,B,C,\zeta)$. 

\end{Thm}

\begin{proof}
(1) From Lemma \ref{L4.1}, we know that $\alpha\in \mathscr E^2_\psi(A,B,C,\zeta)$. From Lemma \ref{L4.2}, we know that $f\Diamond_ig\in \mathscr E_\psi^\bullet(A,B,C,\zeta)$ whenever
$f$, $g\in \mathscr E_\psi^\bullet(A,B,C,\zeta)$. This proves the result.

\smallskip
(2) From Lemma \ref{L3.5}, we know that $(\mathscr C^\bullet_\psi(A,B,C,\zeta),\delta^\bullet)$ is the cochain complex induced by the weak comp algebra $(\mathscr C_\psi^\bullet(A,B,C,\zeta),\Diamond,\alpha)$. Since $(\mathscr E_\psi^\bullet(A,B,C,\zeta),\Diamond,\alpha)$ is   a weak comp subalgebra of $(\mathscr C_\psi^\bullet(A,B,C,\zeta),\Diamond,\alpha)$, the result follows.

\smallskip
(3) From \eqref{cup1} and \eqref{cup2}, we recall that for  $f\in \mathscr C^m_\psi(A,B,C,\zeta)$, 
$g\in \mathscr C^n_\psi(A,B,C,\zeta)$, we have
\begin{equation}\label{cup1z}
\begin{array}{l}
(f\cup g)(c\otimes M)\\ =\zeta\left(\Pi M((1,m+1);(m,m+n))\right)\cdot f(c_1\otimes M((1,1);(m,m))_\psi)\cdot g(c_2^{\psi^m}\otimes M((m+1,m+1);(n,n)))
\end{array}
\end{equation} and 
\begin{equation}\label{cup2z}
\begin{array}{l}
(f\sqcup g)(c\otimes M)\\ =\zeta\left(\Pi M((1,m+1);(m,m+n))\right)\cdot f(c_2\otimes M((1,1);(m,m)))_\psi \cdot g(c_1^\psi \otimes M((m+1,m+1);(n,n)))
\end{array}
\end{equation} Now if $f\in \mathscr E_\psi^m(A,B,C,\zeta)$, applying the condition in \eqref{eqsub} makes it clear that the expressions in \eqref{cup1z} and
\eqref{cup2z} are equal. From the general properties of a weak comp algebra, we know that these cup products descend to cohomology groups and $  \bar{f}\cup \bar{g}=(-1)^{mn}\bar{g}\sqcup \bar{f} $ for cohomology classes $\bar{f}\in H^m(\mathscr E^\bullet_\psi(A,B,C,\zeta))$ and $\bar{g}\in H^n(\mathscr E^\bullet_\psi(A,B,C,\zeta))$. The result is now clear.
\end{proof}

In Theorem \ref{Th3}, we showed that $(\mathscr C_\psi^\bullet(A,B,C,\zeta),\Diamond,\alpha)$ carries the structure of a weak comp algebra. We will now show that the tuple $(\mathscr E_\psi^\bullet(A,B,C,\zeta),\Diamond,\alpha)$ satisfies a stronger condition, i.e., it is a comp algebra in the sense of Gerstenhaber and Schack \cite[$\S$ 4]{GS2}. Further, we will show  that the graded commutative structure on  the cohomology
groups of $\mathscr E_\psi^\bullet(A,B,C,\zeta)$ is part of a Gerstenhaber algebra structure on $H^\bullet(\mathscr E_\psi^\bullet(A,B,C,\zeta))$.   First, we recall the following two definitions.

\begin{defn}\label{Def4.4} (see \cite[$\S$ 4]{GS2}  Let $k$ be a field. A (right)  comp algebra $(V^\bullet,\Diamond,\alpha)$ over $k$ consists of the following data:

\smallskip
(A) A graded vector space $V=\underset{i\geq 0}{\bigoplus}V^i$ and a given element $\alpha\in V^2$

\smallskip
(B) A family $\Diamond$ of operations 
\begin{equation}
\Diamond_i:V^m\otimes V^n\longrightarrow V^{m+n-1} \qquad \forall\textrm{ }i\geq 0
\end{equation} satisfying the following conditions for any $f\in V^m$, $g\in V^n$, $h\in V^p$:

\smallskip
\begin{enumerate}
\item $f\Diamond_ig =0$ if $i>m-1$.

\item  $(f\Diamond_i g)\Diamond_j h=(f\Diamond_j h)\Diamond_{i+p-1}g$ if $j<i$

\item $(f\Diamond_i g)\Diamond_j h=f\Diamond_i(g\Diamond_{j-i}h)$ if $i\leq j<n+i$

\item $(f\Diamond_ig)\Diamond_jh=(f\Diamond_{j-n+1}h)\Diamond_ig$ if $j\geq n+i$

\item $\alpha\Diamond_0\alpha = \alpha\Diamond_1\alpha$.
\end{enumerate}

\end{defn}  

\begin{defn}\label{Def4.5} (see \cite{G1}) Let $k$ be a field. A Gerstenhaber algebra $(V,\cup,[.,.])$ over $k$ consists of the following data:

\smallskip
(A) A graded vector space $V=\underset{i\geq 0}{\bigoplus}V^i$.

\smallskip
(B) A cup product  $\cup: V^m\otimes V^n\longrightarrow V^{m+n}$, $\forall$ $m$, $n\geq 0$ so that $(V,\cup)$ carries the structure of a graded
commutative algebra. 

\smallskip
(C) A degree $-1$ Lie bracket $[.,.]:V^m\otimes V^n\longrightarrow V^{m+n-1}$, $\forall$ $m$, $n\geq 0$ so that $(V,[.,.])$ carries the structure of a graded Lie algebra.

\smallskip
(D) The bracket is a graded derivation for  the cup product, i.e., 
\begin{equation*}
[f,g\cup h]=[f,g]\cup h+(-1)^{m(n+1)}g\cup [f,h]
\end{equation*} for $f\in V^m$, $g\in V^n$ and $h\in V^p$. 

\end{defn}

In particular, given a comp algebra $(V^\bullet,\Diamond,\alpha)$, one may consider the following operations defined on it
\begin{equation}\label{eq4.13h}
\begin{array}{c}
f\cup g:=(\alpha\Diamond_0f)\Diamond_mg  \qquad f\Diamond g:=\underset{i\geq 0}{\sum} (-1)^{i(n-1)}f\Diamond_ig\qquad 
\mbox{$[f,g]:=f\Diamond g-(-1)^{(m-1)(n-1)}g\Diamond f$} \\
\end{array}
\end{equation} for $f\in V^m$ and $g\in V^n$. Then, it follows (see \cite[$\S$ 5]{GS2}) that the operations $\cup$ and $[.,.]$ descend to  cohomology and $(H^\bullet(V),\cup, [.,.])$ carries the structure of a Gerstenhaber algebra. 

\begin{lem}\label{L4.6} For $f\in \mathscr E_\psi^m(A,B,C,\zeta)$, $g\in \mathscr E^n_\psi(A,B,C,\zeta)$ and $h\in \mathscr E_\psi^p(A,B,C,\zeta)$,
we have
\begin{equation*}
(f\Diamond_ig)\Diamond_jh=(f\Diamond_jh)\Diamond_{i+p-1}g 
\end{equation*} for $j<i$. 

\end{lem}

\begin{proof}
We write an
$m+n+p-2$  square  ``upper triangular tensor matrix'' $M$ in $A^{\otimes (m+n+p-2)}\otimes 
B^{\otimes \frac{(m+n+p-2)(m+n+p-3)}{2}}$  as a $5\times 5$ matrix  
\begin{equation}\label{3.6eqzr}
M=\begin{pmatrix}
U(j) & X^{12} & X^{13} & X^{14} & X^{15} \\
1 & U(p) & X^{23} & X^{24} & X^{25} \\
1 & 1 & U(i-j-1) & X^{34} & X^{35} \\
1& 1 & 1 & U(n) & X^{45} \\
1 & 1 & 1 & 1 & U(m-i-1) \\
\end{pmatrix}
\end{equation} where each $U(k)$ in \eqref{3.6eqzr} is a square block of dimension $k$.  For $c\in C$, we now have
\begin{equation*}
\begin{array}{ll}
((f\Diamond_ig)\Diamond_jh)(c\otimes M)&=(f\Diamond_ig)\left(c_1\otimes \begin{pmatrix}
U(j)_\psi &\Pi^r X^{12} & X^{13} & X^{14} & X^{15} \\
1 & h(c_2^{\psi^j}\otimes U(p)) &\Pi^c X^{23} & \Pi^c X^{24} & \Pi^c X^{25} \\
1 & 1 & U(i-j-1) & X^{34} & X^{35} \\
1& 1 & 1 & U(n) & X^{45} \\
1 & 1 & 1 & 1 & U(m-i-1) \\
\end{pmatrix}\right) \\
& = f\left(c_1\otimes \begin{pmatrix}
U(j)_{\psi\psi} &\Pi^r X^{12} & X^{13} & \Pi^rX^{14} & X^{15} \\
1 & h(c_2^{\psi^j}\otimes U(p)_\psi) &\Pi^c X^{23} &\Pi^r \Pi^c X^{24} & \Pi^c X^{25} \\
1 & 1 & U(i-j-1)_\psi &  \Pi^rX^{34} & X^{35} \\
1& 1 & 1 & g(c_3^{\psi^{i+p-1}}\otimes U(n)) &\Pi^c X^{45} \\
1 & 1 & 1 & 1 & U(m-i-1) \\
\end{pmatrix}\right) \\
\end{array}
\end{equation*} where the last equality follows from applying the condition in \eqref{eqsub} to $h\in \mathscr E^p_\psi(A,B,C,\zeta)$. On the other hand, we have
\begin{equation*}
\begin{array}{ll}
((f\Diamond_jh)\Diamond_{i+p-1}g)(c\otimes M)& =(f\Diamond_jh)\left(c_1\otimes \begin{pmatrix}
U(j)_\psi & X^{12} & X^{13} & \Pi^rX^{14} & X^{15} \\
1 & U(p)_\psi & X^{23} & \Pi^rX^{24} & X^{25} \\
1 & 1 & U(i-j-1)_\psi & \Pi^rX^{34} & X^{35} \\
1& 1 & 1 & g(c_2^{\psi^{i+p-1}}\otimes U(n)) & \Pi^c X^{45} \\
1 & 1 & 1 & 1 & U(m-i-1) \\
\end{pmatrix}\right)\\
\end{array}
\end{equation*} The result is now clear.
\end{proof}

\begin{lem} \label{L4.7} For $f\in \mathscr E_\psi^m(A,B,C,\zeta)$, $g\in \mathscr E^n_\psi(A,B,C,\zeta)$ and $h\in \mathscr E_\psi^p(A,B,C,\zeta)$,
we have
\begin{equation*}
(f\Diamond_ig)\Diamond_jh=  (f\Diamond_{j-n+1}h)\Diamond_ig
\end{equation*} for $j\geq i+n$. 

\end{lem}

\begin{proof} We know that $j\geq i+n$. If $j-n+1>m-1$, both sides are already zero. We assume therefore that
$j-n+1\leq m-1$ and hence $m+n-2-j\geq 0$. We write an
$m+n+p-2$  square  ``upper triangular tensor matrix'' $M$ in $A^{\otimes (m+n+p-2)}\otimes 
B^{\otimes \frac{(m+n+p-2)(m+n+p-3)}{2}}$  as a $5\times 5$ matrix  
\begin{equation}\label{3.6eqzrt}
M=\begin{pmatrix}
U(i) & X^{12} & X^{13} & X^{14} & X^{15} \\
1 & U(n) & X^{23} & X^{24} & X^{25} \\
1 & 1 & U(j-i-n) & X^{34} & X^{35} \\
1& 1 & 1 & U(p) & X^{45} \\
1 & 1 & 1 & 1 & U(m+n-j-2) \\
\end{pmatrix}
\end{equation}
 where each $U(k)$ in \eqref{3.6eqzrt} is a square block of dimension $k$. 
 
  For $c\in C$, we now have
 \begin{equation*}
\begin{array}{ll}
((f\Diamond_ig)\Diamond_jh)(c\otimes M)&=(f\Diamond_ig)\left(c_1\otimes \begin{pmatrix}
U(i)_\psi & X^{12} & X^{13} & \Pi^rX^{14} & X^{15} \\
1 & U(n)_\psi & X^{23} & \Pi^rX^{24} & X^{25} \\
1 & 1 & U(j-i-n)_\psi & \Pi^r X^{34} & X^{35} \\
1& 1 & 1 & h(c_2^{\psi^j}\otimes U(p)) & \Pi^cX^{45} \\
1 & 1 & 1 & 1 & U(m+n-j-2) \\
\end{pmatrix}\right) \\\\
& = f\left(c_1\otimes \begin{pmatrix}
U(i)_{\psi\psi} &\Pi^r X^{12} & X^{13} & \Pi^rX^{14} & X^{15} \\
1 & g(c_3^{\psi^i}\otimes U(n))_\psi &\Pi^c X^{23} &\Pi^r \Pi^c X^{24} & \Pi^c X^{25} \\
1 & 1 & U(j-i-n)_\psi &  \Pi^rX^{34} & X^{35} \\
1& 1 & 1 & h(c_2^{\psi^{j-n+1}}\otimes U(p)) &\Pi^c X^{45} \\
1 & 1 & 1 & 1 & U(m+n-j-2) \\
\end{pmatrix}\right) 
\end{array}
\end{equation*} where the last equality follows from applying the condition in \eqref{eqsub} to $g\in \mathscr E^n_\psi(A,B,C,\zeta)$.   On the other hand, we have
\begin{equation*}
\begin{array}{l}
((f\Diamond_{j-n+1}h)\Diamond_{i}g)(c\otimes M)\\ =(f\Diamond_{j-n+1}h)\left(c_1\otimes \begin{pmatrix}
U(i)_\psi & \Pi^rX^{12} & X^{13} & X^{14} & X^{15} \\
1 & g(c_2^{\psi^i}\otimes U(n)) & \Pi^cX^{23} & \Pi^cX^{24} & \Pi^cX^{25} \\
1 & 1 & U(j-i-n) & X^{34} & X^{35} \\
1& 1 & 1 & U(p) & X^{45} \\
1 & 1 & 1 & 1 & U(m+n-j-2) \\
\end{pmatrix}\right)\\
\end{array}
\end{equation*} The result is now clear.
\end{proof}

\begin{Thm}\label{thfn}   Let $B$ be a commutative $k$-algebra and let $(A,B,C,\psi,\zeta)$ be an entwining structure over $B$. Then:

\smallskip
(1)  The tuple 
$(\mathscr E_\psi^\bullet(A,B,C,\zeta),\Diamond,\alpha)$ is the data of a right comp algebra over $k$.
 
\smallskip
(2) The cohomology $(H^\bullet(\mathscr E_\psi^\bullet(A,B,C,\zeta)),\cup,[.,.])$ carries the structure of a Gerstenhaber algebra. 

\end{Thm}

\begin{proof}
The result of (1) is clear from the conditions in Definition \ref{Def4.4} and the results of Lemmas \ref{L3.2}, \ref{L4.6} and \ref{L4.7}. The comp algebra
$(\mathscr E_\psi^\bullet(A,B,C,\zeta),\Diamond,\alpha)$ now carries the operations 
\begin{equation*}
\begin{array}{c}
f\cup g:=(\alpha\Diamond_0f)\Diamond_mg  \qquad f\Diamond g:=\underset{i\geq 0}{\sum} (-1)^{i(n-1)}f\Diamond_ig\qquad 
\mbox{$[f,g]:=f\Diamond g-(-1)^{(m-1)(n-1)}g\Diamond f$} \\
\end{array}
\end{equation*} as explained in \eqref{eq4.13h}. From \cite[$\S$ 5]{GS2}, it now follows that the operations
$\cup$ and $[.,.]$ descend to   cohomology and $(H^\bullet(\mathscr E_\psi^\bullet(A,B,C,\zeta)),\cup,[.,.])$ carries the structure of a Gerstenhaber algebra. 
\end{proof}

\section{Deformation theory}
We consider the map $\zeta: B\longrightarrow A$ such that $\zeta(B)\subseteq Z(A)$. Following \cite{St1},
we see that for each $b\in B$,  the map $\mu_b:A \otimes A \longrightarrow A$ given by $\mu_b(x\otimes y)=\zeta(b)xy$ defines a product structure on $A$. We will now consider deformations of an entwining structure $(A,B,C,\psi,\zeta)$ over $B$.

\begin{defn} \label{Dinfdef} 
Let $(A,B,C,\psi,\zeta)$ be an entwining structure over $B$.  A formal  deformation $(\mu_t,\Delta_t,\psi_t)$ of  $(A,B,C,\psi,\zeta)$ over $B$ consists of the following data:

\smallskip
(1) A family $\mu_t=\{\mu^{(i)}:A\otimes A\otimes B\longrightarrow A\}_{i\geq 1}$ of $k$-linear maps such that the collection
\begin{equation*}
\{\mu_{b,t}: A[[t]]\otimes_{k[[t]]} A[[t]]\longrightarrow A[[t]]\}_{b\in B}
\end{equation*}  of $k[[t]]$-linear maps determined by 
\begin{equation*}
\mu_{b,t}(x\otimes y) =\mu^{(0)}(x \otimes y \otimes b)+\sum_{i=1}^\infty \mu^{(i)}(x\otimes y \otimes b) t^i\qquad x,y\in A
\end{equation*} 
with $\mu^{(0)}(x \otimes y \otimes b)=\zeta(b)xy=\mu_b(x \otimes y)$, satisfies the generalized associativity condition
\begin{equation*}
\mu_{b_1b_2,t}(id \otimes \mu_{b_3,t})=\mu_{b_2b_3,t}(\mu_{b_1,t} \otimes id):A[[t]]\otimes_{k[[t]]} A[[t]]\otimes_{k[[t]]} A[[t]]\longrightarrow A[[t]]
\end{equation*}
for any $b_1,b_2,b_3 \in B$. In particular, for each $b \in B$, we have a $k$-linear map $\mu^{(i)}_b:A \otimes A \longrightarrow A$ given by $\mu^{(i)}_b(x \otimes y)=\mu^{(i)}(x \otimes y \otimes b)$.

\smallskip
(2) A $k[[t]]$-linear map $\Delta_t: C[[t]]\longrightarrow C[[t]]\otimes_{k[[t]]} C[[t]]$ of the form
\begin{equation*}
\Delta_t (c)= \Delta(c)+ \sum_{i=1}^\infty \Delta^{(i)}(c)t^i\qquad c\in C
\end{equation*} making $C[[t]]$ into a coassociative coalgebra over $k[[t]]$. For $c\in C$, we write $\Delta^{(i)}(c)=\sum c_{(i1)}\otimes c_{(i2)}\in C\otimes C$. 

\smallskip
(3) A family $\psi_t=\{\psi^{(i)}:C\otimes A\longrightarrow A\otimes C\}_{i\geq 1}$ of $k$-linear maps such that 
the $k[[t]]$-linear map determined by
\begin{equation*}
\psi_t: C[[t]]\otimes_{k[[t]]} A[[t]]\longrightarrow A[[t]]\otimes_{k[[t]]} C[[t]] \qquad \psi_t(c\otimes x)=\psi(c\otimes x)+\sum_{i=1}^\infty \psi^{(i)}(c\otimes x)t^i
\end{equation*} for $c\in C$ and $x\in A$ satisfies 
\begin{eqnarray}
\psi_t(id \otimes \mu_{b,t})=(\mu_{b,t} \otimes id)(id \otimes \psi_t)(\psi_t \otimes  id):C[[t]]\otimes_{k[[t]]} A[[t]]\otimes_{k[[t]]} A[[t]]\longrightarrow A[[t]]\otimes_{k[[t]]} C[[t]] \label{eq 6.1s}\\
(id \otimes \Delta_t)\psi_t=(\psi_t \otimes  id)(id \otimes \psi_t)(\Delta_t \otimes id):C[[t]]\otimes_{k[[t]]} A[[t]]\longrightarrow A[[t]]\otimes_{k[[t]]} C[[t]]\otimes_{k[[t]]} C[[t]] \label{eq 6.2x}
\end{eqnarray} For $c\otimes a\in C\otimes A$, we write $\psi^{(i)}(c\otimes a)=\sum a_{\psi^{(i)}}\otimes c^{\psi^{(i)}}$.

\end{defn}

\begin{defn}\label{equivdeform} Let $(A,B,C,\psi,\zeta)$ be an entwining structure over $B$. We will say that two deformations  $(\mu_t,\Delta_t,\psi_t)$ and  $(\mu'_t,\Delta'_t,\psi'_t)$ of  $(A,B,C,\psi,\zeta)$ over $B$ are equivalent if there exists a pair $(\alpha,\gamma)$ of $k[[t]]$-linear isomorphisms $\alpha: A[[t]]\longrightarrow A[[t]]$ and $\gamma: C[[t]]
\longrightarrow C[[t]]$ such that 

\smallskip
(1) The map $\alpha: A[[t]]\longrightarrow A[[t]]$ is determined by 
\begin{equation*}
\alpha(x)=x+\sum_{i=1}^\infty \alpha^{(i)}(x)t^i\qquad x\in A
\end{equation*} for $k$-linear maps $\{\alpha^{(i)}:A\longrightarrow A\}_{i\geq 1}$  and satisfies
\begin{equation*}
\alpha (\mu_{b,t}(x\otimes y))=\mu'_{b,t}(\alpha(x)\otimes \alpha(y)) \qquad x,y\in A
\end{equation*} for each $b\in B$. 

\smallskip
(2) The map $\gamma: C[[t]]\longrightarrow C[[t]]$  is determined by 
\begin{equation*}
\gamma(c)=c+\sum_{i=1}^\infty \gamma^{(i)}(c)t^i\qquad c\in C
\end{equation*} for $k$-linear maps $\{\gamma^{(i)}:C\longrightarrow C\}_{i\geq 1}$  and  is a morphism of coalgebras, i.e., $(\gamma\otimes \gamma)\circ \Delta_t=\Delta'_t\circ \gamma$. 

\smallskip
(3) The maps $\alpha$ and $\gamma$ are compatible with respect to $\psi_t$ and $\psi'_t$, i.e., 
\begin{equation*}
\psi'_t\circ (\gamma\otimes \alpha)=(\alpha\otimes \gamma)\circ \psi_t:C[[t]]\otimes_{k[[t]]} A[[t]]\longrightarrow A[[t]]\otimes_{k[[t]]} C[[t]]
\end{equation*}

\end{defn}
 
 \smallskip
In order to understand the deformations of $(A,B,C,\psi,\zeta)$ over $B$, we now introduce a double complex extending that of Brzezi\'{n}ski \cite[$\S$ 6]{Brz2}. 

\begin{lem}[see \cite{Brz2}]\label{vert} Let $(A,C,\psi)$ be an entwining structure and let $(V, \rho_L^V,\rho_R^V)$ be a $C$-bicomodule. Then, we have a complex $A^\bullet_\psi(C,V)$ given by 
\begin{equation*}
A^n_\psi(C,V)=Hom(V,A \otimes C^n)
\end{equation*}
with the differential $\bar{\delta}^n:Hom(V,A \otimes C^n) \longrightarrow Hom(V,A \otimes C^{n+1})$ defined by
\begin{equation*}
\bar{\delta}^nf=(\psi \otimes C^n) \circ (C \otimes f) \circ \rho_L^V+ \sum\limits_{i=1}^n (-1)^i  (A \otimes C^{i-1} \otimes \Delta \otimes C^{n-i}) \circ f~ + (-1)^{n+1}~(f \otimes C) \circ \rho_R^V
\end{equation*}
\end{lem}

\begin{lem}\label{j} Let $(A,B,C,\psi,\zeta)$ be an entwining structure over $B$. 
For any $A$-bimodule $(M,\triangleright_L^M,\triangleright_R^M)$ and $C$-bicomodule $(V,\rho_L^V,\rho_R^V)$, the linear maps 
\begin{eqnarray*}
&j:Hom(A^m \otimes B^\frac{m(m-1)}{2}, M) \longrightarrow Hom(C \otimes A^m \otimes B^\frac{m(m-1)}{2}, M) \qquad f \mapsto \varepsilon \otimes f\\
&\bar{j}: Hom(V,C^n) \longrightarrow Hom(V, A \otimes C^n) \qquad g \mapsto 1 \otimes g
\end{eqnarray*}
determine morphisms of complexes.
\end{lem}
\begin{proof}
This follows from Proposition \ref{P2.4} and  \cite[Lemma 3.2]{Brz2}.
\end{proof}

For each $n\geq 0$, we note that $A\otimes C^{ n}$ is an $A$-bimodule with actions determined by
\begin{equation}
a'\cdot (a\otimes c^1\otimes ...\otimes c^n)\cdot a''=a'aa''_{\psi^n}\otimes c^{1^\psi}\otimes ... \otimes c^{n^\psi} 
\end{equation} As such, for each fixed $n$, we have Hochschild differentials
\begin{equation}\label{hordef}
\delta^m: Hom(C \otimes A^m \otimes B^\frac{m(m-1)}{2}, A \otimes C^n)\longrightarrow Hom(C \otimes A^{m+1} \otimes B^\frac{m(m+1)}{2}, A \otimes C^n)
\end{equation} On the other hand, for each fixed $m$, it is easily verified that the coactions $\rho^m_L$ and $\rho^m_R$ as in \eqref{4.3prq} make each $C \otimes A^m \otimes B^\frac{m(m-1)}{2}$ into a $C$-bicomodule. Hence, we can form a complex as in Lemma \ref{vert} with coefficients in $C \otimes A^m \otimes B^\frac{m(m-1)}{2}$ and having differentials
\begin{equation}\label{verdef}
\bar{\delta}^n: Hom(C \otimes A^m \otimes B^\frac{m(m-1)}{2},A\otimes C^n)\longrightarrow Hom( C \otimes A^m \otimes B^\frac{m(m-1)}{2},A\otimes C^{n+1})
\end{equation}

\begin{thm}\label{bicmpl}
We have a bicomplex $C^{\bullet,\bullet}(A,B,C,\psi,\zeta)$ determined by
\begin{equation*}
C^{m,n}(A,B,C,\psi,\zeta):=Hom(C \otimes A^m \otimes B^\frac{m(m-1)}{2}, A \otimes C^n) \qquad m, n \geq 1
\end{equation*}
\begin{equation*}
C^{m,0}(A,B,C,\psi,\zeta):=Hom(A^m \otimes B^\frac{m(m-1)}{2}, A) \qquad C^{0,n}(A,B,C,\psi,\zeta):=Hom(C,C^n)
\end{equation*}
with the horizontal and vertical differentials $d_h$ and $d_v$ respectively given by  
\begin{equation}\label{difform}
\begin{array}{l}
d_h^{m,n}:=(-1)^m\delta^m: C^{m,n}(A,B,C,\psi,\zeta) \longrightarrow C^{m+1,n}(A,B,C,\psi,\zeta)\\
d_v^{m,n}:=(-1)^m\bar{\delta}^n: C^{m,n}(A,B,C,\psi,\zeta) \longrightarrow C^{m,n+1}(A,B,C,\psi,\zeta) \\
d_h^{0,n}:=\delta^0 \circ \bar{j}:Hom(C,C^n) \longrightarrow Hom(C \otimes A, A \otimes C^n) \\  d_v^{0,n}=\bar\delta:Hom(C,C^n) \longrightarrow Hom(C,C^{n+1})\\
d_v^{m,0}:=(-1)^m\bar{\delta}^0 \circ j :Hom(A^m \otimes B^\frac{m(m-1)}{2}, A) \longrightarrow Hom(C \otimes A^m \otimes B^\frac{m(m-1)}{2}, A \otimes C)\\
 d_h^{m,0}: =(-1)^m\delta :Hom(A^m \otimes B^\frac{m(m-1)}{2}, A) \longrightarrow Hom(A^{m+1} \otimes B^\frac{m(m+1)}{2}, A) \\
 \end{array}
 \end{equation}for $m,n \geq 1$.
\end{thm}

\begin{proof}
By Proposition \ref{P2.3} and Lemma \ref{vert}, we know that $d_h^2=0$ and $d_v^2=0$.
For $m,n \geq 1$, we write 
\begin{eqnarray*}
\delta^m=\sum\limits_{i=0}^{m+1} (-1)^i \delta_i: Hom(C \otimes A^m \otimes B^{\frac{m(m-1)}{2}}, A \otimes C^n) \longrightarrow Hom(C \otimes A^{m+1} \otimes B^{\frac{m(m+1)}{2}}, A \otimes C^n) \\
\bar{\delta}^n=\sum\limits_{j=0}^{n+1} (-1)^j \bar{\delta}_j:Hom(C \otimes A^m \otimes B^{\frac{m(m-1)}{2}}, A \otimes C^n) \longrightarrow Hom(C \otimes A^m \otimes B^{\frac{m(m-1)}{2}}, A \otimes C^{n+1})
\end{eqnarray*}

 We will now prove that $\delta_i \bar{\delta}_j=\bar{\delta}_j \delta_i$ for all $0 \leq i \leq m+1$ and $0 \leq j \leq n+1$. It follows immediately from the definition that $\delta_i$ and $\bar{\delta}_j$ commute for $1 \leq i \leq m$ and $1 \leq j \leq n$. Moreover, for any $f \in Hom(C \otimes A^m\otimes B^{\frac{m(m-1)}{2}}, A \otimes C^n)$ with $m$, $n\geq 1$, we have
 
 \begin{eqnarray*}
&(\delta_0\bar{\delta}_0f)\left(c\bigotimes \begin{pmatrix}
a_1 & b_{12} & b_{13} & b_{14} & ...&  b_{1,m+1} \\
1 & a_2 & b_{23}  & b_{24} &  ...& b_{2,m+1} \\
1 & 1 & a_3  & b_{34} &  ...& b_{3,m+1} \\
. & . & .  & . &  ...&  .\\
1 & 1 & 1 & 1 &  ...& b_{m,m+1}\\
1 & 1 & 1 & 1 &  ...& a_{m+1} \\
\end{pmatrix} \right)
= \zeta\left( \underset{k=2}{\overset{m+1}{\prod}}b_{1j}\right) a_{1\psi}\cdot (\bar{\delta}_0f)\left(c^\psi\bigotimes \begin{pmatrix}
 a_2 & b_{23}  &  ...& b_{2,m} & b_{2,m+1} \\
 1 & a_3  &  ... & b_{3,m} & b_{3,m+1} \\
 . & .   &  ... & . & .\\
 1 &  1 &  ...&  a_{m} & b_{m,m+1}\\
 1 &  1 &  ...&  1 & a_{m+1} \\
\end{pmatrix}\right)\\
&= \zeta\left( \underset{k=2}{\overset{m+1}{\prod}}b_{1j}\right) a_{1\psi\psi}\cdot (\psi \otimes C^n)\left( {(c_1)}^\psi \otimes f \left({(c_2)}^\psi \bigotimes \begin{pmatrix}
 a_2 & b_{23}   &  ... & b_{2,m} & b_{2,m+1} \\
 1 & a_3  &  ...& b_{3,m} & b_{3,m+1} \\
 . & .   &  ... & . & .\\
 1 & 1  &  ...&  a_{m} & b_{m,m+1}\\
 1 & 1  &  ...&  1 & a_{m+1} \\
\end{pmatrix}\right)\right)\\
 \end{eqnarray*}
We set \begin{eqnarray*}
f \left({(c_2)}^\psi \bigotimes \begin{pmatrix}
 a_2 & b_{23}   &  ... & b_{2,m} & b_{2,m+1} \\
 1 & a_3  &  ...& b_{3,m} & b_{3,m+1} \\
 . & .   &  ... & . & .\\
 1 & 1  &  ...&  a_{m} & b_{m,m+1}\\
 1 & 1  &  ...&  1 & a_{m+1} \\
\end{pmatrix}\right)=\sum a_s \otimes d_s \in A \otimes C^n
 \end{eqnarray*}
 Then,
 \begin{eqnarray*}
&(\delta_0\bar{\delta}_0f)\left(c\bigotimes \begin{pmatrix}
a_1 & b_{12} & b_{13} & b_{14} & ...&  b_{1,m+1} \\
1 & a_2 & b_{23}  & b_{24} &  ...& b_{2,m+1} \\
1 & 1 & a_3  & b_{34} &  ...& b_{3,m+1} \\
. & . & .  & . &  ...&  .\\
1 & 1 & 1 & 1 &  ...& b_{m,m+1}\\
1 & 1 & 1 & 1 &  ...& a_{m+1} \\
\end{pmatrix} \right)
=\sum  \zeta\left( \underset{k=2}{\overset{m+1}{\prod}}b_{1j}\right) a_{1\psi\psi}\cdot (\psi \otimes C^n)\left( {(c_1)}^\psi \otimes a_s \otimes d_s\right)\\
&=\sum  \zeta\left( \underset{k=2}{\overset{m+1}{\prod}}b_{1j}\right) a_{1\psi\psi}a_{s\psi} \otimes {(c_1)}^{\psi\psi}  \otimes d_s=\sum  \zeta\left( \underset{k=2}{\overset{m+1}{\prod}}b_{1j}\right) {(a_{1\psi}a_s)}_\psi \otimes {(c_1)}^{\psi}  \otimes d_s\\
&=\sum  \zeta\left( \underset{k=2}{\overset{m+1}{\prod}}b_{1j}\right) (\psi \otimes C^n)(c_1 \otimes a_{1\psi}a_s \otimes  d_s)\\ 
& = (\psi \otimes C^n)\left(c_1 \otimes (\delta_0f)\left(c_2 \bigotimes \begin{pmatrix}
a_1 & b_{12} & b_{13}  & ...&  b_{1,m+1} \\
1 & a_2 & b_{23}   &  ...& b_{2,m+1} \\
1 & 1 & a_3   &  ...& b_{3,m+1} \\
. & . & .  &   ...&  .\\
1 & 1 & 1 &   ...& b_{m,m+1}\\
1 & 1 & 1 &   ...& a_{m+1} 
\end{pmatrix} \right)\right)
=(\bar{\delta_0}\delta_0f)\left(c\bigotimes \begin{pmatrix}
a_1 & b_{12} & b_{13}  & ...&  b_{1,m+1} \\
1 & a_2 & b_{23}  &  ...& b_{2,m+1} \\
1 & 1 & a_3  &  ...& b_{3,m+1} \\
. & . & .  &  ...&  .\\
1 & 1 & 1  &  ...& b_{m,m+1}\\
1 & 1 & 1  &  ...& a_{m+1} 
\end{pmatrix} \right)
\end{eqnarray*}
Similarly, we can verify that $\delta_0\bar{\delta}_{n+1}=\bar{\delta}_{n+1}\delta_0$, $\delta_{m+1}\bar{\delta}_0=\bar{\delta}_0\delta_{m+1}$ and  $\delta_{m+1}\bar{\delta}_{n+1}=\bar{\delta}_{n+1}\delta_{m+1}$. It follows that $(\delta\circ \bar\delta)(f) =(\bar\delta\circ \delta)(f)$ for  $f\in Hom(C \otimes A^m\otimes B^{\frac{m(m-1)}{2}}, A \otimes C^n)$ with $m$, $n\geq 1$.
Using Proposition \ref{P2.3} and Lemma \ref{j}, we also see that $\delta \circ ( \bar{\delta} \circ j)=( \bar{\delta} \circ j )\circ \delta$ and $\bar{\delta} \circ (\delta \circ \bar{j})=( \delta \circ \bar{j} )\circ   \bar{\delta}$. This shows that $d_h \circ d_v +d_ v \circ d_h=0$.
\end{proof}

The result of Proposition \ref{bicmpl} gives us a bicomplex

\small{
\[
\begin{CD}
  @.           Hom(A,A) @>-\delta>>           Hom(A^2 \otimes B,A) @>\delta>>         Hom(A^3 \otimes B^3,A)   @> -\delta>>                  @.                  @.                 \\
      @.                 @V-\bar{\delta} \circ j VV                @V\bar{\delta} \circ j  VV              @V-\bar{\delta} \circ j  VV                              @.                  @.            @.   \\
Hom(C,C)@         >\delta \circ \bar{j}>>      Hom(C \otimes A, A \otimes C) @>-\delta>>          Hom(C \otimes A^2 \otimes B, A \otimes C) @> \delta>>          Hom(C \otimes A^3 \otimes B^3, A \otimes C)        @>-\delta >>                  @.                  @.                 \\
      @V \bar{\delta} VV                    @V- \bar{\delta} VV                @V  \bar{\delta} VV              @V  -\bar{\delta} VV                 @.                  @.            @.   \\  
Hom(C,C^2) @         > \delta \circ \bar{j} >>       Hom(C \otimes A, A \otimes C^2) @>-\delta>>    Hom(C \otimes A^2 \otimes B, A \otimes C^2) @>\delta>>         Hom(C \otimes A^3 \otimes B^3, A \otimes C^2)      @>-\delta>>       @.                  @.                 \\
      @V \bar{\delta} VV                            @V - \bar{\delta} VV                     @V \bar{\delta} VV             @V  -\bar{\delta}VV             @.             @.            @.   \\
      @.                           @.           @.     @.    @.                  @.
\end{CD}
\]
}

We can now form a total complex $(Tot^\bullet(A,B,C,\psi,\zeta),D)$ with differential $D=d_h+d_v$ and 
\begin{align*}
Tot^p(A,B,C,\psi,\zeta)=Hom(A^p \otimes B^{\frac{p(p-1)}{2}}, A) \bigoplus\limits_{m,n \geq 1, m+n=p}Hom(C \otimes A^m \otimes  B^{\frac{m(m-1)}{2}}, A \otimes C^n) \bigoplus Hom(C,C^p)
\end{align*}

The cohomology groups of $(Tot^\bullet(A,B,C,\psi,\zeta),D)$ will be denoted by $H^\bullet(A,B,C,\psi,\zeta)$.

\begin{defn}
An infinitesimal deformation of an entwining structure $(A,B,C,\psi,\zeta)$ over $B$ is a deformation $(\mu_{t}, \Delta_t, \psi_t)$ of $(A,B,C,\psi,\zeta)$ modulo $t^2$. Then, for each $b \in B$, we write
\begin{eqnarray*}
\mu_{b,t}=\mu_b + \mu^{(1)}_bt, \qquad \Delta_t=\Delta + \Delta^{(1)}t, \qquad \psi_t=\psi+\psi^{(1)}t
\end{eqnarray*}  
Here, $\mu^{(1)} \in Hom(A^2 \otimes B,A)$, $\psi^{(1)} \in Hom(C \otimes A, A \otimes C)$ and  $\Delta^{(1)} \in Hom(C,C^2)$.
\end{defn}

\begin{thm}
Let $(\mu_{t}, \Delta_t, \psi_t)$ be an infinitesimal deformation of an entwining structure $(A,B,C,\psi,\zeta)$ over $B$. Then, $(\mu^{(1)}, \psi^{(1)},  \Delta^{(1)}) \in Hom(A^2 \otimes B,A) \bigoplus Hom(C \otimes A, A \otimes C) \bigoplus Hom(C,C^2)$ is a 2-cocycle in  $(Tot^\bullet(A,B,C,\psi,\zeta),D)$.
\end{thm} 
\begin{proof}
We will show that $D(\mu^{(1)}, \psi^{(1)}, \Delta^{(1)})=0$, i.e.,

\begin{equation*}
\left(\delta \mu^{(1)}, (\bar{\delta} \circ j)\mu^{(1)} - \delta\psi^{(1)},  - \bar{\delta}\psi^{(1)} + (\delta \circ \bar{j})\Delta^{(1)}, \bar{\delta}\Delta^{(1)}\right) =0
\end{equation*}
Using the deformation theory in \cite{St1} and \cite{GS2}, we know that $\delta \mu^{(1)}=0$ and $\bar{\delta}\Delta^{(1)} =0$. Therefore, it remains to show that
\begin{eqnarray}
&(\bar{\delta} \circ j)\mu^{(1)} - \delta\psi^{(1)}=0 \qquad \text{in} \quad  Hom(C \otimes A^2 \otimes B, A \otimes C)\label{2coc1}\\
& (\delta \circ \bar{j})\Delta^{(1)} - \bar{\delta}\psi^{(1)}=0   \qquad \text{in} \quad Hom(C \otimes A, A \otimes C^2)\label{2coc2}
\end{eqnarray}
Using the condition in \eqref{eq 6.1s} and equating the coefficients of $t$, we get that for each $b \in B$,
\begin{equation*}
\psi(C \otimes \mu^{(1)}_b)+\psi^{(1)}(C \otimes \mu_b)
-(\mu_b \otimes C)(A \otimes \psi)(\psi^{(1)} \otimes  A)
-(\mu_b \otimes C)(A \otimes \psi^{(1)})(\psi \otimes  A)- (\mu^{(1)}_b \otimes C)(A \otimes \psi)(\psi \otimes  A)=0
\end{equation*}
Now, for any 
$ c \otimes a_1 \otimes a_2 \otimes b_{12} \in C \otimes A^2 \otimes B$, we have
\begin{equation*}
\begin{array}{ll}
&((\bar{\delta} \circ j)\mu^{(1)})  \left(c \otimes a_1 \otimes a_2 \otimes b_{12}\right)   - (\delta\psi^{(1)})\left(c \otimes a_1 \otimes a_2 \otimes b_{12}\right) =(\bar{\delta} (\varepsilon \otimes \mu^{(1)}))  \left(c \otimes a_1 \otimes a_2 \otimes b_{12} \right)   - (\delta\psi^{(1)})\left(c \otimes a_1 \otimes a_2 \otimes b_{12}\right) \hspace{3.8cm}\\
& \quad =\left(\psi  \circ (C \otimes \varepsilon \otimes \mu^{(1)})\right) \left( c_1 \otimes c_2 \otimes a_1 \otimes a_2 \otimes b_{12}\right) - (\varepsilon \otimes \mu^{(1)} \otimes C)\left(c_1 \otimes  a_{1\psi} \otimes a_{2\psi} \otimes b_{12} \otimes {c_2}^{\psi\psi}\right)\\
& \quad \quad -\quad  \zeta(b_{12})a_{1\psi} \cdot \psi^{(1)}(c^\psi \otimes a_2) + \psi^{(1)}(c \otimes a_1a_2 \zeta(b_{12}))-\psi^{(1)}(c \otimes a_1) \cdot a_2\zeta(b_{12})\\
& \quad =\psi  \left( c \otimes \mu^{(1)}(a_1 \otimes a_2 \otimes b_{12})\right) - \mu^{(1)}(a_{1\psi} \otimes a_{2\psi} \otimes b_{12}) \otimes {c}^{\psi\psi}-\quad  \zeta(b_{12})a_{1\psi} \cdot \psi^{(1)}(c^\psi \otimes a_2) \hspace{1.3cm} \\
& \quad \quad +~ \psi^{(1)}(c \otimes a_1a_2 \zeta(b_{12}))-\psi^{(1)}(c \otimes a_1) \cdot a_2\zeta(b_{12})\\
&\quad =  \psi(C \otimes  \mu^{(1)}_{b_{12}}) (c \otimes a_1 \otimes a_2) - (\mu^{(1)}_{ b_{12}} \otimes C)(A\otimes \psi)(\psi \otimes  A)\left(c \otimes a_1 \otimes a_2 \right)-(\mu_{b_{12}} \otimes C)(A \otimes \psi^{(1)})(\psi \otimes  A)(c \otimes a_1 \otimes a_2) \\
&\quad  +  \psi^{(1)}(C \otimes \mu_{b_{12}})(c \otimes a_1 \otimes  a_2)-(\mu_{b_{12}} \otimes C)(A \otimes \psi)(\psi^{(1)} \otimes  A)(c \otimes a_1 \otimes a_2)=0
\end{array}
\end{equation*}
This proves \eqref{2coc1}. Now using the condition in \eqref{eq 6.2x} and equating the coefficients of $t$, we obtain
\begin{equation*}
(A \otimes \Delta)\psi^{(1)} + (A \otimes \Delta^{(1)})\psi-(\psi \otimes C)(C \otimes \psi)(\Delta^{(1)} \otimes A)-(\psi \otimes C)(C \otimes \psi^{(1)})(\Delta \otimes A)-(\psi^{(1)} \otimes C)(C \otimes \psi)(\Delta \otimes A)=0
\end{equation*}
For any $c \otimes a \in C \otimes A$, we have
\begin{equation*}
\begin{array}{ll}
&\left((\delta \circ \bar{j})\Delta^{(1)}\right) (c \otimes a)- (\bar{\delta}\psi^{(1)})(c \otimes a)=\left(\delta (1 \otimes \Delta^{(1)})\right) (c \otimes a)- (\bar{\delta}\psi^{(1)})(c \otimes a)\\
&=a_\psi \cdot  (1 \otimes \Delta^{(1)})(c^\psi)- \left((1 \otimes \Delta^{(1)})(c)\right) \cdot a- (\psi \otimes C)(C \otimes \psi^{(1)})(c_1 \otimes c_2 \otimes a)+ (A \otimes \Delta)\psi^{(1)}(c \otimes a)-(\psi^{(1)} \otimes C)(c_1 \otimes a_\psi \otimes {c_2}^\psi)\\
&=a_\psi  \otimes \Delta^{(1)}(c^\psi)-a_{\psi\psi}  \otimes {c_{11}}^\psi \otimes {c_{12}}^\psi -a_{\psi^{(1)}\psi} \otimes {c_1}^\psi \otimes {c_2}^{\psi^{(1)}}+ (A \otimes \Delta)\psi^{(1)}(c \otimes a) -(\psi^{(1)} \otimes C)(C \otimes \psi)(\Delta \otimes A)(c \otimes a)\\
&=(A \otimes  \Delta^{(1)})\psi(c \otimes a)-(\psi \otimes C)(C \otimes \psi)(\Delta^{(1)} \otimes A)(c \otimes a)-(\psi \otimes C)(C \otimes \psi^{(1)})(\Delta \otimes A)(c \otimes a)\\
&\quad + (A \otimes \Delta)\psi^{(1)}(c \otimes a)-(\psi^{(1)} \otimes C)(C \otimes \psi)(\Delta \otimes A)(c \otimes a)\\
&=0
\end{array}
\end{equation*}
This proves \eqref{2coc2}.
\end{proof}

\begin{thm}
Let $(\mu_t,\Delta_t,\psi_t)$ and  $(\mu'_t,\Delta'_t,\psi'_t)$ be two infinitesimal deformations of an entwining structure $(A,B,C,\psi,\zeta)$ over $B$.  Then, 
$(\mu_t,\Delta_t,\psi_t)$ and  $(\mu'_t,\Delta'_t,\psi'_t)$ are equivalent if and only if 
the  corresponding 2-cocycles $(\mu^{(1)}, \psi^{(1)}, \Delta^{(1)})$ and $(\mu'^{(1)}, \psi'^{(1)}, \Delta'^{(1)})$ are cohomologous.
\end{thm}
\begin{proof}
Let $\alpha:A[[t]] \longrightarrow A[[t]]$ and $\gamma:C[[t]] \longrightarrow C[[t]]$ be the $k[[t]]$-linear isomorphisms defining the equivalence between $(\mu_t,\Delta_t,\psi_t)$ and  $(\mu'_t,\Delta'_t,\psi'_t)$.  Therefore, 
\begin{eqnarray}
 {\alpha} \circ \mu_{b,t}=\mu'_{b,t} \circ ( {\alpha} \otimes  {\alpha})\label{1} \\
( {\gamma} \otimes  {\gamma}) \circ \Delta_t=\Delta'_t\circ  {\gamma}\label{2}\\
\psi'_t\circ ( {\gamma} \otimes {\alpha})=( {\alpha} \otimes {\gamma})\circ \psi_t \label{3}
\end{eqnarray}
for each $b \in B$. We will now prove that 
\begin{equation}\label{6.8nb}
(\mu^{(1)}-\mu'^{(1)}, \psi^{(1)}-\psi'^{(1)}, \Delta^{(1)}- \Delta'^{(1)})=D(-\alpha^{(1)},-\gamma^{(1)})=(\delta \alpha^{(1)}, (\bar{\delta} \circ j)\alpha^{(1)}- (\delta \circ \bar{j})\gamma^{(1)}, -\bar{\delta} \gamma^{(1)})
\end{equation}

Using the conditions in \eqref{1} and \eqref{2}, we see that
\begin{equation*}
\mu^{(1)}-\mu'^{(1)} =\delta \alpha^{(1)} \qquad  \Delta^{(1)}- \Delta'^{(1)}=- \bar{\delta} \gamma^{(1)}
\end{equation*}
Further, for any $c \otimes a \in C \otimes A$, we have
\begin{align*}
&-((\bar{\delta} \circ j)\alpha^{(1)})(c \otimes a)+ ((\delta \circ \bar{j})\gamma^{(1)})(c \otimes a)=-(\bar{\delta} (\varepsilon \otimes \alpha^{(1)}))(c \otimes a)+ (\delta(1 \otimes \gamma^{(1)}))(c \otimes a)\\
&=-(\psi \circ (C \otimes \varepsilon \otimes \alpha^{(1)}))(c_1 \otimes c_2 \otimes a)+ ( \varepsilon \otimes \alpha^{(1)} \otimes C)(c_1 \otimes a_\psi \otimes {c_2}^\psi)+ a_\psi \cdot (1 \otimes  \gamma^{(1)})(c^\psi)- \left((1 \otimes  \gamma^{(1)})(c)\right)\cdot a\\
&=-\psi(C \otimes  \alpha^{(1)})(c \otimes a)+( \alpha^{(1)} \otimes C)\psi(c \otimes a)+( A \otimes  \gamma^{(1)})\psi(c \otimes a)-\psi ( \gamma^{(1)} \otimes A)(c \otimes a)\\
&=\psi'^{(1)}(c \otimes a)-\psi^{(1)}(c \otimes a)
\end{align*}
The last equality follows by using \eqref{3}. This proves \eqref{6.8nb}.

\smallskip
Conversely, let $(f,g,h)$ and $(f',g',h')$ be 2-cocycles  in $Hom(A^2,A) \bigoplus Hom(C \otimes A, A \otimes C) \bigoplus Hom(C,C^2)$ such that
\begin{equation*}
(f-f',g-g',h-h')=D(f_1,h_1) 
\end{equation*}
Then, it may be verified that the linear isomorphims $\alpha:=A-f_1t:A[[t]] \longrightarrow A[[t]]$ and $\gamma:=C - h_1t:C[[t]] \longrightarrow C[[t]]$ define an equivalence
between the infinitesimal deformations corresponding respectively to $(f,g,h)$ and $(f',g',h')$. 
\end{proof}

We have therefore obtained the following result.

\begin{Thm}\label{MThinf}
Let $(A,B,C,\psi,\zeta)$ be an entwining structure over $B$. Then, there is a one-to-one correspondence between equivalence classes of infinitesimal deformations of  $(A,B,C,\psi,\zeta)$ and the  cohomology group $H^2(A,B,C,\psi,\zeta)$.
\end{Thm}

\begin{Thm}\label{obsthmi}
Let $(\mu_t,\Delta_t,\psi_t)$ be an $n$-truncated deformation of  an entwining structure $(A,B,C,\psi,\zeta)$ over $B$. Then, $(\mu_t,\Delta_t,\psi_t)$ can be extended to an $(n+1)$-truncated deformation of $(A,B,C,\psi,\zeta)$ over $B$ if and only if the  obstruction given by the $3$-cochain  in $Tot^\bullet(A,B,C,\psi,\zeta)$
\begin{equation*}
Obs^{n+1}=(Obs_{A,B}^{n+1},Obs^{n+1}_{\mu_t,\psi_t}, Obs^{n+1}_{\Delta_t,\psi_t}, Obs^{n+1}_C) \in Hom(A^3 \otimes B^3,A) \oplus Hom(C \otimes A^2 \otimes B, A \otimes C)  \oplus Hom(C \otimes A, A \otimes C^2) \oplus Hom(C,C^3)
\end{equation*}
is a coboundary, where 
\begin{align*}
Obs_{A,B}^{n+1}\begin{pmatrix}
a_1& b_{12}&b_{13}\\
1&a_2& b_{23}\\
1&1&a_3
\end{pmatrix}&:=\sum\limits_{\mbox{\tiny $\begin{array}{c}i,j < n+1\\ i+j=n+1\end{array}$}} \left(\mu^{(j)}_{b_{13}b_{23}}(\mu^{(i)}_{b_{12}} \otimes A)-\mu^{(i)}_{b_{12}b_{13}}(A \otimes \mu^{(j)}_{b_{23}})\right)(a_1,a_2,a_3)\\
Obs^{n+1}_{\mu_t,\psi_t}
\left(c \bigotimes \begin{pmatrix}
a_1& b_{12}\\
1&a_2
\end{pmatrix} \right)
&:=\left( \sum\limits_{\mbox{\tiny $\begin{array}{c}i,j,l <n+1\\ i+j+l=n+1\end{array}$}}(\mu^{(i)}_{b_{12}} \otimes C)(A \otimes \psi^{(j)})(\psi^{(l)} \otimes A)-\sum\limits_{\mbox{\tiny $\begin{array}{c}i,j<n+1\\ i+j=n+1\end{array}$}}\psi^{(i)}(C \otimes \mu^{(j)}_{b_{12}})\right)(c \otimes a_1 \otimes a_2)\\
Obs_{\Delta_t,\psi_t}^{n+1}
&:=\sum\limits_{\mbox{\tiny $\begin{array}{c}i,j,l <n+1\\ i+j+l=n+1\end{array}$}}(\psi^{(i)} \otimes C)(C \otimes \psi^{(j)})(\Delta^{(l)} \otimes A)-\sum\limits_{\mbox{\tiny $\begin{array}{c}i,j<n+1\\ i+j=n+1\end{array}$}} (A \otimes \Delta^{(i)})\psi^{(j)}\\
Obs_C^{n+1}&:=\sum\limits_{i,j <n+1, i+j=n+1}\left((\Delta^{(j)} \otimes C)\Delta^{(i)}- (C \otimes \Delta^{(j)})\Delta^{(i)}\right)
\end{align*} for any $c\in C$,  $a_1,a_2,a_3 \in A$ and  $b_{12},b_{13},b_{23} \in B$.
\end{Thm}

\begin{proof} By definition, an  $n$-truncated formal  deformation of an entwining structure  $(A,B,C,\psi,\zeta)$ over $B$ is a deformation $(\mu_t,\Delta_t,\psi_t)$ of  $(A,B,C,\psi,\zeta)$ modulo $t^{n+1}$. By equating coefficients of powers of $t$, we obtain for each $1 \leq k \leq n$, 

\begin{equation}\label{obs1}
\left(\sum\limits_{i,j< k, i+j=k} \left(\mu^{(j)}_{b_{13}b_{23}}(\mu^{(i)}_{b_{12}} \otimes A) -\mu^{(i)}_{b_{12}b_{13}}(A \otimes \mu^{(j)}_{b_{23}})\right)\right)(a_1,a_2,a_3)=(\delta \mu^{(k)})\begin{pmatrix}
a_1& b_{12}&b_{13}\\
1&a_2& b_{23}\\
1&1&a_3
\end{pmatrix}
\end{equation}

\smallskip
\begin{equation}\label{obs2}
\begin{array}{l}
\left(\sum\limits_{i,j,l <k, i+j+l=k}(\mu^{(i)}_{b_{12}} \otimes C)(A \otimes \psi^{(j)})(\psi^{(l)} \otimes A)-\sum\limits_{i,j<k, i+j=k} \psi^{(i)}(C \otimes \mu^{(j)}_{b_{12}})\right)(c \otimes a_1 \otimes a_2)\\\\
\qquad=((\bar{\delta} \circ j)\mu^{(k)} -\delta \psi^{(k)})\left(c \bigotimes \begin{pmatrix}
a_1 & b_{12}\\
1&a_2
\end{pmatrix} \right)
\end{array}
\end{equation}

\smallskip
\begin{equation}\label{obs3}
\begin{array}{l}
\sum\limits_{i,j,l <k, i+j+l=k}(\psi^{(i)} \otimes C)(C \otimes \psi^{(j)})(\Delta^{(l)} \otimes A) -\sum\limits_{i,j<k, i+j=k} (A \otimes \Delta^{(i)})\psi^{(j)}=-\bar{\delta}\psi^{(k)} +(\delta \circ \bar{j})\Delta^{(k)} 
\end{array}
\end{equation}

\smallskip
\begin{equation}\label{obs4}
\sum\limits_{i,j<k, i+j=k}\left((\Delta^{(j)} \otimes C)\Delta^{(i)}-(C \otimes \Delta^{(j)})\Delta^{(i)}\right)=\bar{\delta} \Delta^{(k)}
\end{equation}
for any $c\in C$,  $a_1,a_2,a_3 \in A$ and  $b_{12},b_{13},b_{23} \in B$. From standard arguments in deformation theory (see, for instance, \cite{Brz2}, \cite{St2}), it may be verified that $DObs^{n+1}=0$. From \eqref{obs1}, \eqref{obs2}, \eqref{obs3} and \eqref{obs4}, it is clear that in order to lift $(\mu_t,\Delta_t,\psi_t)$ to a deformation modulo $t^{n+2}$, we should have a tuple $(\mu^{(n+1)}, \psi^{(n+1)},\Delta^{(n+1)})\in Tot^2(A,B,C,\psi,\zeta)$ such that
\begin{equation*}
Obs^{n+1}=D(\mu^{(n+1)}, \psi^{(n+1)},\Delta^{(n+1)})
\end{equation*} This proves the result.
\end{proof}


\begin{bibdiv}
	\begin{biblist}
	
	\bib{Abu}{article}{
   author={Abuhlail, J. Y.},
   title={Dual entwining structures and dual entwined modules},
   journal={Algebr. Represent. Theory},
   volume={8},
   date={2005},
   number={2},
   pages={275--295},
 
}
	



	
\bib{BBR}{article}{
   author={Balodi, M.},
   author={Banerjee, A.},
   author={Ray, S.}
   title={Entwined modules over linear categories and Galois extensions},
   journal={Israel Journal of Mathematics},
   volume={ 241},
   date={2021},
   pages={623--692},
}

\bib{BBN}{article}{
   author={Balodi, M.},
   author={Banerjee, A.},
   author={Naolekar, A.},
   title={BV-operators and the secondary Hochschild complex},
   journal={C. R. Math. Acad. Sci. Paris},
   volume={358},
   date={2020},
   number={11-12},
   pages={1239--1258},
}	
	
	
	
	
	
	
	\bib{Brz1}{article}{
   author={Brzezi\'{n}ski, T.},
   author={Majid, S.},
   title={Coalgebra bundles},
   journal={Comm. Math. Phys.},
   volume={191},
   date={1998},
   number={2},
   pages={467--492},
}

\bib{Brz2.5}{article}{
   author={Brzezi\'{n}ski, T.},
   title={On modules associated to coalgebra Galois extensions},
   journal={J. Algebra},
   volume={215},
   date={1999},
   number={1},
   pages={290--317},
   
}

\bib{Brz2}{article}{
   author={Brzezi\'{n}ski, T.},
   title={The cohomology structure of an algebra entwined with a coalgebra},
   journal={J. Algebra},
   volume={235},
   date={2001},
   number={1},
   pages={176--202},
}

\bib{Brz3}{article}{
   author={Brzezi\'{n}ski, T.},
   title={The structure of corings: induction functors, Maschke-type
   theorem, and Frobenius and Galois-type properties},
   journal={Algebr. Represent. Theory},
   volume={5},
   date={2002},
   number={4},
   pages={389--410},
  
}



\bib{BCT}{article}{
   author={Bulacu, D.},
   author={Caenepeel, S.},
   author={Torrecillas, B.},
   title={Frobenius and separable functors for the category of entwined
   modules over cowreaths, II: applications},
   journal={J. Algebra},
   volume={515},
   date={2018},
   pages={236--277},

}

\bib{CaDe}{article}{
   author={Caenepeel, S.},
   author={De Groot, E.},
   title={Modules over weak entwining structures},
   conference={
      title={New trends in Hopf algebra theory},
      address={La Falda},
      date={1999},
   },
   book={
      series={Contemp. Math.},
      volume={267},
      publisher={Amer. Math. Soc., Providence, RI},
   },
   date={2000},
   pages={31--54},
}



\bib{St3}{article}{
   author={Corrigan-Salter, B. R.},
   author={Staic, M. D.},
   title={Higher-order and secondary Hochschild cohomology},
   journal={C. R. Math. Acad. Sci. Paris},
   volume={354},
   date={2016},
   number={11},
   pages={1049--1054},

}



\bib{G1}{article}{
   author={Gerstenhaber, M.},
   title={The cohomology structure of an associative ring},
   journal={Ann. of Math. (2)},
   volume={78},
   date={1963},
   pages={267--288},
}

\bib{GerSh}{article}{
   author={Gerstenhaber, M.},
   author={Schack, S. D.},
   title={A Hodge-type decomposition for commutative algebra cohomology},
   journal={J. Pure Appl. Algebra},
   volume={48},
   date={1987},
   number={3},
   pages={229--247},

}

\bib{GS2}{article}{
   author={Gerstenhaber, M.},
   author={Schack, S. D.},
   title={Algebras, bialgebras, quantum groups, and algebraic deformations},
   conference={
      title={Deformation theory and quantum groups with applications to
      mathematical physics},
      address={Amherst, MA},
      date={1990},
   },
   book={
      series={Contemp. Math.},
      volume={134},
      publisher={Amer. Math. Soc., Providence, RI},
   },
   date={1992},
   pages={51--92},
}

\bib{Jia}{article}{
   author={Jia, L.},
   title={The sovereign structure on categories of entwined modules},
   journal={J. Pure Appl. Algebra},
   volume={221},
   date={2017},
   number={4},
   pages={867--874},
   issn={0022-4049},

}

\bib{KPS}{article}{
   author={Kluge, L.},
   author={Paal, E.},
   author={Stasheff, J.},
   title={Invitation to composition},
   journal={Comm. Algebra},
   volume={28},
   date={2000},
   number={3},
   pages={1405--1422},
}

\bib{St4}{article}{
   author={Laubacher, J.},
   author={Staic, M.~ D.},
   author={Stancu, A.},
   title={Bar simplicial modules and secondary cyclic (co)homology},
   journal={J. Noncommut. Geom.},
   volume={12},
   date={2018},
   number={3},
   pages={865--887},
}




\bib{Schbg}{article}{
   author={Schauenburg, P.},
   title={Doi-Koppinen Hopf modules versus entwined modules},
   journal={New York J. Math.},
   volume={6},
   date={2000},
   pages={325--329},
}


\bib{St1}{article}{
   author={Staic, M. D.},
   title={Secondary Hochschild cohomology},
   journal={Algebr. Represent. Theory},
   volume={19},
   date={2016},
   number={1},
   pages={47--56},
}

\bib{St2}{article}{
   author={Staic, M. D.},
   author={Stancu, A.},
   title={Operations on the secondary Hochschild cohomology},
   journal={Homology Homotopy Appl.},
   volume={17},
   date={2015},
   number={1},
   pages={129--146},

}
	
\end{biblist}

\end{bibdiv}
\end{document}